\documentclass[a4paper,10pt,reqno]{amsart}
\usepackage[utf8]{inputenc}
\usepackage{amssymb}
\usepackage{amsthm}
\usepackage{tikz}
\usepackage{graphicx}
\usepackage{mathrsfs}
\usepackage{epsf}
\usepackage{pdfpages}
\usepackage{bookmark,hyperref}
\usepackage{stmaryrd}
\usepackage[colorinlistoftodos]{todonotes}
\hypersetup{
     colorlinks=true,
     linkcolor=blue,
     filecolor=blue,
     citecolor = blue,
     urlcolor=cyan,
     }
\usepackage{pstricks}
\graphicspath{{figures/}}

\newtheorem{theorem}{Theorem}[section]
\newtheorem{lemma}{Lemma}[section]
\newtheorem{corollary}{Corollary}[section]
\newtheorem{proposition}{Proposition}[section]
\newtheorem*{theorem*}{Theorem}
\newtheorem*{proposition*}{Proposition}

\newtheoremstyle{mythm}%
{3pt}% Space above
{3pt}% Space below
{}% Body font 
{}% Indent amount
{\bfseries}% Theorem head font
{.}% Punctuation after theorem head 
{.5em}% Space after theorem head 
{}%

\theoremstyle{mythm}
\newtheorem{definition}{Definition}[section]

\newtheorem{remark}{Remark}[section]

\numberwithin{equation}{section}

\newcommand{\interior}[1]{{\kern0pt#1}^{\mathrm{o}}}

\def\N{\mathbb{N}}

\def\R{\mathbb{R}}

\def\set#1{\left\{\, #1 \,\right\}}

\def\norm #1{\left\| \,#1\, \right\|}
\def\snorm #1{\| \,#1\, \|}
\def\inner #1#2{\langle \,#1,#2\, \rangle}

\def\H{\mathbb{H}}
\def\S{\mathbb{S}}

\def\calA{\mathcal{A}}
\def\calB{\mathscr{B}}
\def\calC{\mathcal{C}}
\def\calD{\mathscr{D}}
\def\calE{\mathcal{E}}
\def\calF{\mathcal{F}}

\def\calH{\mathscr{H}}

\def\calK{\mathcal{K}}

\def\calM{\mathscr{M}}

\def\calP{\mathcal{P}}

\def\calS{\mathscr{S}}
\def\calT{\mathcal{T}}
\def\calU{\mathcal{U}}
\def\calV{\mathcal{V}}

\def\sm{\mathfrak{a}}

\def\cone{\mathscr C}

\title[Uniqueness of Busemann functions]
{Uniqueness of hyperbolic Busemann functions\\
in the Newtonian $N$-body problem}
\author{Ezequiel Maderna}
\address{CIMAT -- Centro de Investigación en Matemáticas\\
Guanajuato, México}
\email{ezequiel.maderna@cimat.mx}
\author{Andrea Venturelli}
\address{LMA -- Laboratoire de Mathématiques d'Avignon , France}
\email{andrea.venturelli@univ-avignon.fr}
\thanks{Supported by Laboratorio Del Plata
(IRL-2030 of CNRS) and
ANR Project CoSyDy.}
\date{\today}

\begin{document}

\begin{abstract}
For the $N$-body problem we prove
that any two hyperbolic rays having the same limit shape
define the same Busemann function.
We localize a region of differentiability for these functions,
of which we know that they are
viscosity solutions of the stationary Hamilton-Jacobi equation.
As a first corollary, we deduce that every hyperbolic motion
of the $N$-body problem must become,
after some time, a calibrating curve for the Busemann function
associated to its limit shape.
This implies that every hyperbolic motion
of the $N$-body problem is eventually a minimizer, that is,
it must contain a geodesic ray of the Jacobi-Maupertuis metric.
Since the viscosity solutions of the Hamilton-Jacobi equation
are almost everywhere differentiable,
we also deduce the generic uniqueness of geodesic rays
with a given limit shape without collisions.
That is to say, if the limit shape is given, then,
for almost every initial configuration the geodesic ray is unique.
\end{abstract}

\maketitle
\tableofcontents

\section{Introduction}

In this paper we deal with the classical $N$-body
problem of Celestial Mechanics, and the global solutions
of the Hamilton-Jacobi of this problem.
In order to give an explicit formulation we must recall
some usual notations. Let $E$ be an Euclidean space where
$N$ punctual and positive masses $m_1,\dots,m_N$ are moving,
because of his mutual gravitational attraction.
The forces derive from the Newtonian potential in $E^N$,
whose value at a configuration
$x=(r_1,\dots,r_N)$ is given by
\[
U(x)=\sum_{i<j}m_im_jr_{ij}^{-1}
\]
where $r_{ij}=\norm{r_i-r_j}_E$ is the distance in $E$
between the bodies $i$ and $j$.
The open and dense set 
\[
\Omega=\set{x=(r_1,\dots,r_N)\in E^N
\mid r_i\neq r_j \textrm{ for all } i\neq j}
\]
is called the set of configurations without collisions.
A very important simplification of the formulation
is obtained by introducing the \emph{mass inner product}
in $E^N$. It is defined as follows:
\[
\inner{(r_1,\dots,r_N)}{(s_1,\dots,s_N)}=
m_1\inner{r_1}{s_1}_E+\dots+m_N\inner{r_N}{s_N}_E
\]
This inner product allow us to write Newton's equations
synthetically in the form $\ddot x=\nabla U(x)$.
Moreover, the Lagrangian function in $T\Omega$
can simply be written as
\begin{equation}
\label{def.eq:Lagrangian}
L(x,v)=\frac{1}{2}\norm{v}^2+U(x)
\end{equation}
and the Hamiltonian in $T^*\Omega$ as
\begin{equation}
\label{def.eq:Hamoiltonian}
H(x,p)=\frac{1}{2}\norm{p}^2-U(x)
\end{equation}
since the natural identification $T_x\Omega\simeq E^N$
induces a dual norm in $T_x^*\Omega$. Here $\norm{\ }$ denotes the euclidean norm associated 
to the mass inner product $\inner{}{}$.
For our purposes, it is convenient to set $L(x,v)=+\infty$ and
$H(x,p)=-\infty$ whenever $x\in\calC=E^N\setminus\Omega$
(the set of configurations with collisions).
We are interested in the variational formulation
because most of the motions that we will consider
will be obtained as minimizers of the Lagrangian action.
Given two configurations $x,y\in E^N$ and $\tau>0$ we
write $\calC(x,y,\tau)$ for the set of absolutely continuous
curves $\gamma:[a,b]\to E^N$ such that $\gamma(a)=x$,
$\gamma(b)=y$ and $b-a=\tau$.
The Lagrangian action of such a curve is denoted, as usual
\[
\calA_L(\gamma)=\int_a^b L(\gamma,\dot\gamma)\,dt=
\int_a^b\frac{1}{2}\norm{\dot\gamma}^2\,dt+
\int_a^b U(\gamma)\,dt
\] 
and the minimal action to go from $x$ to $y$ in time $\tau$
is defined as
\[
\phi(x,y,\tau)=
\min\set{\calA_L(\gamma)\mid \gamma\in\calC(x,y,\tau)}\,.
\]
It should be said that the fact that the Lagrangian
is singular, in no way precludes the possibility
of considering the Lagrangian action of
absolutely continuous curves passing through collisions.
Eventually this action can be infinite.
\begin{definition}[The Jacobi-Maupertuis distance]
\label{def:jm-distance}
For a non-negative energy level $h\geq 0$,
the Jacobi-Maupertuis distance
between two configurations $x$ and $y$ is defined as 
\[
\phi_h(x,y)=
\min\set{\calA_{L+h}(\gamma)\mid \gamma\in\calC(x,y)}=
\min\set{\phi(x,y,\tau)+h\tau \mid \tau>0}
\]
where $\calC(x,y)=\cup_{\,\tau>0}\,\calC(x,y,\tau)$
is the set of all absolutely continuous curves going
from configuration $x$ to configuration $y$ without time constraint, and if $\gamma\in \calC(x,y)$ is defined in the interval $[a,b]$ 
we denote
\begin{equation}\label{eq:Action.L+h}
\calA_{L+h}(\gamma)=\int_a^b (\;L(\gamma(t),\dot \gamma(t))+h\;)dt.    
\end{equation}
\end{definition}
\begin{remark}
\label{rem-jm}
The fact that $\phi_h$ is a distance on $E^N$ is a consequence of the definition and of Lemma 4.4 in \cite{MaVe}.
It is also well known that when the dimension of $E$ is bigger or equal to $2$ the minimum value
in $\phi(x,y,\tau)$, and also in $\phi_h(x,y)$ when $x\neq y$,
is reached on curves without collisions in intermediate times,
and therefore they are true motions of the $N$-body problem.
The existence follows from Tonelli's theory,
and the absence of collisions is guaranteed by the famous
Marchal's Theorem (see \cite{Che1, Mar}).
Moreover, it is immediate to verify that a curve realizing the minimum in $\phi_h$ is a motion of the 
$N$-body problem with energy given by $h$. 
\end{remark}
\subsection{Hyperbolic motions}
\label{ss:Hyp-mot-bsm-fnct}
We deal with \emph{hyperbolic} motions of a
Newtonian $N$-body problem
in the Euclidean space $E$ of dimension $d\geq 2$.
These are those defined over an unbounded interval $t\geq t_0$,
and such that $x(t)=ta+o(t)$ as $t\to+\infty$,
where $x(t)=(r_1(t),\dots,r_N(t))\in E^N$ is the configuration given
by the positions of the bodies at time $t$
and $a=(a_1,\dots,a_N)$ is a configuration without collisions.
Therefore the mutual distances satisfy
$r_{ij}(t)=\norm{r_i(t)-r_j(t)}=t\,a_{ij}+o(t)$,
where $a_{ij}=\norm{a_i-a_j}>0$.
It follows from Chazy's asymptotic analysis of
hyperbolic expansions (see Remark \ref{remark:Chazy} below) 
that the energy constant $h$ of an hyperbolic motion
is given by $h=\frac{1}{2}\norm{a}^2$, in particular it is strictly positive. 

\subsection{Geodesic rays and Busemann functions}
\label{ss: ged-rays-busemann}

The equation of motion of the $N$-body problem
defines a Hamiltonian flow taking place
on the cotangent space $T^*\Omega$, with Hamiltonian given by the function $H$ defined in 
(\ref{def.eq:Hamoiltonian}), which is also called the 
energy function.
It is well known that for every $h\geq 0$, the set $\Omega$ can be
endowed with a Riemannian metric for which motions of energy $h$
are geodesic curves.
This is the Jacobi-Maupertuis metric given by
\[
j_h=2\,(h+U)\,g_m
\]
where $U$ is the Newtonian potential and $g_m$ is the
flat metric on $E^N$ induced by the mass inner product.
Every motion of the Newtonian $N$-body problem with
energy constant $h$ is a geodesic of the Jacobi-Maupertuis
metric $j_h$, but not necessarily parametrized by the arc length.
Actually, if $\gamma: [t_0,t_1]\to \Omega$ is an absolutely continuous curve, 
the Jacobi-Maupertuis arc length of $\gamma$ 
is smaller or equal to the action $\calA_{L+h}(\gamma)$, with equality if and only if $H(\gamma(t),\dot{\gamma}(t)=h$ 
for almost every $t\in [t_0,t_1]$.
Consequently \emph{minimizing geodesics}
joining two given configurations are
in correspondence with the free time minimizers of $\calA_{L+h}$,
in the space of curves joining these two configurations.
In our previous works \cite{MV,MaVe} we have shown the abundance
of geodesic rays for these Riemannian metrics, in the sense
that we can find them with arbitrary initial configuration.
Moreover, for the positive energy case $h>0$, these geodesic rays
can be taken as hyperbolic motions $x(t)=ta+o(t)$
with arbitrary initial configuration $x(0)=x_0\in E^N$ and
even with arbitrary limit shape $a\in\Omega$.
These existence results have been extended
by Polimeni and Terracini \cite{PolTer} through a new approach,
finding in the same framework all these geodesic rays,
as well as others of partially hyperbolic type. See also the work of Burgos \cite{Bur}.

The Riemannian manifold $(\Omega,j_h)$ is clearly noncompact,
which motivates the study of the corresponding Busemann functions
associated to these geodesic rays.
Recently, by studying the value function
of the renormalized action functional previously introduced in \cite{PolTer},
Berti, Polimeni and Terracini recover in \cite{BePoTe} the Busemann functions
which are the object of the present article.
They prove several regularity results about them, in particular concerning the
Hausdorff dimension of the set of singularities.

Recall that a \emph{length space} is an arc-connected metric space
in which the arc length distance coincides with
the original distance. That is to say, the distance between two points
is precisely the infimum of all the lengths of curves joining them.
If moreover, the distance between any two points is the length
of some curve joining them, we say that the metric space is
a \emph{geodesic space}.
We also recall that in a length space, a curve defined
over a half-line $[t_0,+\infty)$ is a \emph{geodesic ray}
if the restriction to any compact interval $[t_0,t]$
is length minimizing.

\vspace{5pt}

\begin{definition}[Busemann functions]
\label{def:Busemann.metric}
Let $(M,d)$ be a given length space,
and let $\gamma:[t_0,+\infty)\to M$ be a geodesic ray.
The Busemann function associated to $\gamma$ is %$b_\gamma:M\to\R$ is the function
defined as the limit
\[
b_\gamma(x)=\lim_{t\to +\infty}\;\;
d(x,\gamma(t))-d(\gamma(t),\gamma(t_0))
=\lim_{t\to +\infty}\;\;
d(x,\gamma(t))-\;\textrm{length}(\gamma\mid_{[t_0,t]})\,.
\]
\end{definition}

\vspace{5pt}

It is straightforward to see that the previous limit exist,
and that the convergence is uniform on compact subsets.
Since the notion of Busemann function is central to this paper,
let us explain this with more details.
For $t\geq t_0$, let us denote $u_{\gamma,t}$ the function
\begin{equation} 
\label{eq:u-gamma-t}
u_{\gamma,t}(x)=d(x,\gamma(t))-d(\gamma(t),\gamma(t_0)).
\end{equation}
Now by triangular inequality,
and using that $\gamma$ is a geodesic ray,
we observe that for any $t_0<t<s$ and $x\in M$ we have
\begin{eqnarray*}
d(x,\gamma(s))&\leq& d(x,\gamma(t))+d(\gamma(t),\gamma(s))\\
&\leq&d(x,\gamma(t))+d(\gamma(s),\gamma(t_0))-d(\gamma(t),\gamma(t_0))   
\end{eqnarray*}
and we conclude that $u_{\gamma,s}\leq u_{\gamma,t}$.
On the other hand, also by triangular inequality, we have that
\[
u_{\gamma,t}(x)=
d(x,\gamma(t))-d(\gamma(t),\gamma(t_0))\geq
-d(x,\gamma(t_0)).
\]
Therefore the limit exists for every $x\in M$,
and moreover, since the functions $u_{\gamma,t}$ are
$1$-Lipschitz, the Busemann function
$b_\gamma=\lim_{\,t\to\infty} \,u_{\gamma,t}$
is also $1$-Lipschitz,
and the convergence is uniform on compact subsets.

The simplest example of a Busemann function can be obtained
by considering a Euclidean space, in which geodesic rays
are nothing more than half-lines, that we can parametrize
as $\gamma(t)=tv+x_0$ for $t\geq 0$. Since the Busemann
function does not depends on the parametrization of the ray,
we can assume that $\gamma$ is parametrized by its arc length,
that is $\norm{v}=1$. In this case we get
\[
u_{\gamma,t}(x)=\norm{x-(tv+x_0)}-t
\]
and the corresponding Busemann function is the affine function
given by
\[
b_\gamma(x)=-\inner{x-x_0}{v}\,.
\]

In order to clarify the meaning of the results we will present,
it is very useful to note the following facts that occur
in the previous trivial example.
\begin{enumerate}
    \item Given two geodesic rays, the corresponding Busemann functions
    differ by a constant if and only if the rays have the same direction.
    \item For every ray $\gamma(t)=tv+x_0$, the Busemann function $b_\gamma$
    is differentiable and its gradient is constant $\nabla b_\gamma (x)=-v$.
    Therefore every geodesic ray with the same direction as $\gamma$
    is obtained \emph{by reversing the parametrization} of a gradient line
    of $b_\gamma$.
\end{enumerate}

Busemann functions have been widely used in the geometric study
of Hadamard manifolds, where they have additional properties
like convexity and differentiability.
In particular, their role in the study of rigidity phenomena
of hyperbolic manifolds is noteworthy.
If we think the hyperbolic space $\H^n$ as an open ball of $\R^n$
endowed with a metric of constant negative curvature, then
the geodesic rays have a limit for the Euclidean distance in $\R^n$,
which is a point in the boundary sphere $\S^{n-1}$ of the ball.
In this case, two geodesic rays define the same Busemann function
-- up to a constant -- if and only if they have the same limit in
the sphere. Note that they are unbounded as curves in $\H^n$.
Moreover, they have the same limit in $\R^n$ if and only if
they remain at a bounded distance in $\H^n$.
These facts fail to be true if the space is not simply connected,
or if it is not negatively curved.
That is, two divergent rays can define the same Busemann function
\cite{DalPeiSam}, or two rays which remain at a bounded distance
can have Busemann functions whose difference is not constant,
see the infinite ladder example in \cite[p. 526]{MaVe}.

\begin{definition}
Given $\gamma : [t_0,+\infty)\to E^N$ an absolutely continuous curve. We say that $\gamma$ is a hyperbolic geodesic ray if 
$\gamma\left|_{(t_0,+\infty)}\right.$ is a hyperbolic motion of the $N$-body problem and for any $t>t_0$ we have 
\begin{equation}
\label{eq:A_L+h=phi_h}
\calA_{L+h}(\gamma\mid_{[t_0,t]})=\phi_h(\gamma(t_0),\gamma(t)),
\end{equation}
where $h>0$  is the energy of the motion $\gamma$.
\end{definition}
As we have already noticed in Remark \ref{rem-jm}, if $\gamma : [t_0,+\infty)\to E^N$ satisfies  (\ref{eq:A_L+h=phi_h}) for some $h>0$, then 
for all $t>t_0$ we have that $\gamma(t)$  is a solution of the $N$-body problem with energy $h$, and in particular $\gamma(t)\in\Omega$. 
In \cite[Proposition 5.1, p. 539]{MaVe} we have shown that $\phi_h$ is the 
completion of the Riemannian distance on $\Omega$ associated to the Jacobi metric $j_h$. It follows that $(E^N,\phi_h)$ 
is a geodesic space, and  each hyperbolic geodesic ray is a genuine geodesic ray of the space $(E^N,\phi_h)$. However, while a geodesic ray is invariant by reparametrization, condition 
(\ref{eq:A_L+h=phi_h}) is not. Actually, condition  (\ref{eq:A_L+h=phi_h}) fixes a parametrization of $\gamma$  such that $H(\gamma(t),\dot{\gamma}(t))=h$ for all $t>t_0$.

Geodesic rays of non-negative energy levels
in the Newtonian $N$-body problem have been extensively
studied in recent years.
First the critical case $h=0$ was studied, 
in particular in the works of Barutello, Terracini and Verzini
\cite{BaTeVe1,BaTeVe2}, in the work of the authors
\cite{MV}, as well as in the work of da Luz and the first author
\cite{daLMad} where it is proved that such motions are  
completely parabolic.
More recently,
this result was generalized by Burgos and the first author
in \cite{BuMa},
proving for the case of energy $h>0$
that every geodesic ray is an expansive motion.
Moreover, either the motion is hyperbolic,
or it must be broken down
into parabolic clusters moving away from each other
in a hyperbolic way, that is, of partially hyperbolic type (see Appendix \ref{app:More-Hyp}).

In the Newtonian $N$-body problem,
the study of the Busemann functions of the non-negative
energy levels endowed with their corresponding
Jacobi-Maupertuis metrics starts in \cite{Mad1}.
There, Busemann functions of the planar Kepler problem
are explicitly computed for the critical energy level $h=0$,
taking as geodesic ray a radial parabolic ejection.
One interesting observation resulting of this explicit determination
is the non-differentiability of these global solutions,
showing the need to consider weak solutions,
and especially viscosity solutions.
The extension to the general $N$-body problem
was done by Percino and Sánchez-Morgado in
\cite{PerSan}. They consider the Busemann functions
associated with parabolic homothetic motions
by minimal configurations,
and show that these functions are viscosity solutions.
Furthermore, they show that all their calibrating curves
are also parabolic motions with the same asymptotic direction.

We recall that a homothetic motion, that is of the form
$x(t)=\lambda(t)x_0$, only occurs through
a central configuration $x_0$,
and that the radial component $\lambda(t)$ must be taken
being a solution of a Kepler problem
$\ddot\lambda=-U(x_0)\lambda^{-2}$.
Given the central configuration $x_0$ there is,
up to translation of time,
only one choice of $\lambda(t)$ giving rise to a
parabolic motion for $t\to +\infty$.
A configuration $x_0\neq 0$ is said to be minimal if
it is a minimizer of $U$ restricted to the sphere
$\set{x\in E^N\mid \norm{x}=\norm{x_0}}$.
Minimal configurations always exist, and they are
central configurations with the additional property
that the associated homothetic parabolic motion is
a geodesic ray of the critical Jacobi-Maupertuis metric $j_0$.
A priori there could be non minimal central configurations
whose parabolic homothetic motion is a geodesic ray,
see \cite{Mad2}.

The first main result in this paper concerns Busemann functions
produced by hyperbolic geodesic rays.
\begin{theorem}
\label{theorem:1.uniqueness.Busemann}
Given $h>0$ and two hyperbolic geodesic rays
$\gamma_1(t)$ and $\gamma_2(t)$,
the following are equivalent.
\begin{enumerate}
    \item[(1)] The Euclidean distance $\norm{\gamma_1(t)-\gamma_2(t)}$ is bounded.
    \item[(2)] The Jacobi-Maupertuis distance $\phi_h(\gamma_1(t),\gamma_2(t))$ is bounded.
    \item[(3)] The two hyperbolic geodesic rays have the same limit shape.
    \item[(4)] The corresponding Busemann functions differ by a constant.
\end{enumerate}
\end{theorem}
The proof of Theorem \ref{theorem:1.uniqueness.Busemann} will be done in section \ref{s: proofs}.

\begin{remark}[Chazy's estimate]
\label{remark:Chazy}
In his celebrated work of 1922 \cite{Cha2}, Jean~Chazy proved
for hyperbolic motions the asymptotic expansion as $t\to +\infty$
\[
x(t)=ta-\log(t)\nabla U(a)+r(1/t,\log(t)/t)    
\]
where $r(z,w)$ is an analytic function in a neighbourhood
of $(0,0)$. It follows that the limit of $\dot x(t)$
exists and coincides with the limit shape $a$. That is,
\[
a\;\;=\;\;\lim_{t\to +\infty}\frac{x(t)}{t}\;\;=
\;\;\lim_{t\to +\infty}\dot x(t)\,.
\]
It also follows from Chazy's estimate that the first
and the third condition
in Theorem 1 above are equivalent to the following one:
\[
\norm{x_1(t)-x_2(t)}=o(t)
\textit{ as }t\to +\infty\,.
\]
\end{remark}

\begin{remark}
The fact that two geodesic rays whose distance remains bounded
define the same Busemann function (up to a constant) is, in general, not true.
It is very well known that this property holds in manifolds of negative curvature,
but this is not our case when the number of bodies is $N\geq 3$.
So Theorem \ref{theorem:1.uniqueness.Busemann} actually shows that the phenomenon can take place
in a context in which the metric is asymptotically flat.
\end{remark}

As a result of Theorem \ref{theorem:1.uniqueness.Busemann},
the Busemann function associated with a given configuration without collisions
configuration turns out to be well-defined.
The following definition allows us to eliminate
the indeterminacy of the constant.

\begin{definition}
[Busemann function associated with a configuration]
\label{def:Busemann.a}
For any given configuration without collisions $a\in\Omega$,
we call \emph{the Busemann function} of $a$
the function $b_a:E^N\to\R$ defined as
\[
b_a(x)=b_\gamma(x)-b_\gamma(0)
\]
where $\gamma$ is any hyperbolic geodesic ray having
configuration $a$ as limit shape.
\end{definition}

It is also common to find in the literature another
way to lift the indeterminacy of the constant.
This consists of considering, for a geodesic ray $\gamma : [t_0,+\infty)\to M$ in a length space $(M,d)$, the so-called Busemann
functions of two variables,
namely those associated with each geodesic ray by the limit
\begin{eqnarray*}
b_\gamma(x,y)&=&\lim_{t\to +\infty}\;\;
d(x,\gamma(t))-d(\gamma(t),y)\\
&=&\lim_{t\to +\infty}\;\;
u_{\gamma,t}(x)-u_{\gamma,t}(y)\\
&=&\;b_\gamma(x)-b_\gamma(y)\, ,
\end{eqnarray*}
where $u_{\gamma,t}(x)$ and $u_{\gamma,t}(y)$ are defined as in
(\ref{eq:u-gamma-t}).
Thus, Theorem \ref{theorem:1.uniqueness.Busemann}
says that for any two hyperbolic rays with the same limit shape,
the corresponding Busemann functions of two variables agree,
and Definition \ref{def:Busemann.a} above can be given by
\[
b_a(x)=b_\gamma(x,0)
\]
where $\gamma$ is any hyperbolic geodesic ray
whose limit shape is configuration $a\in\Omega$.

Our next result concerns all the hyperbolic motions.
\begin{theorem}
\label{thm:diff-busemann}
For any given hyperbolic motion $\gamma(t)=ta+o(t)$
defined for $t\geq t_0$, there exists $t_1>t_0$,
such that the restriction of $\gamma$ to $[t_1,+\infty)$  is a hyperbolic geodesic ray. Moreover, for all $t\geq t_1$
the Busemann function $b_a$ is differentiable
at $\gamma(t)$  and it satisfies
$\dot \gamma(t)=-\nabla b_a(\gamma(t))$.
\end{theorem}

\begin{remark}
It is worth emphasizing that throughout all this article,
the inner product in configuration space is the mass inner product.
Both gradients and Hessian operators are  taken with respect
to this inner product.
\end{remark}

Theorem \ref{thm:diff-busemann}   is reminiscent of \cite{MoecMontgMorg},
where it is proved, for the planar three-body problem,
that any parabolic solution with equilateral limit shape
is eventually a zero-minimizer, that is, a free time minimizer
for the lagrangian action.
That is, the property of being eventually minimizing
is obtained under a condition of minimality of the
central configuration that asymptotically dominates the parabolic expansion
(recall that the Lagrange configuration is a global minimizer
of the Newtonian potential relatively to the size of the configuration).
This strongly contrasts with our result,
since we do not require any hypothesis about the limit shape of the motion,
nor any other type: all hyperbolic motions are eventually minimizing.

\section{Hyperbolic viscosity solutions}
\label{sec:hypviscsol}

\subsection{Hyperbolic solutions to the Hamilton-Jacobi equation}
We will deduce Theorem \ref{theorem:1.uniqueness.Busemann} from a more general result concerning a special class of viscosity 
solutions of the Hamilton Jacobi equation
$H(x,d_xu)=h$. In \cite{MaVe} we highligth a mechanism to construct viscosity solutions
of $H(x,d_xu)=h$, where $H$ is the Hamiltonian of the system,
and $h>0$ is a chosen positive energy level. Here we will show that those viscosity solutions are unique, up to an additive constant. 
Let us start with some definitions, that are quite standard in weak KAM theory.
\begin{definition}[Dominated function]
Let $h>0$. A function $u: E^N\to \R$ is said to be dominated (by $L+h$) if for any absolutely continuous curve $\gamma: [a,b]\to E^N$ we have 
\[
u(\gamma(b))-u(\gamma(a))\leq \calA_{L+h}(\gamma).
\]
\end{definition}
Clearly this condition is equivalent to have
\[
u(y)-u(x) \leq \phi_h(x,y)
\]
for any pair of configurations $x,y\in E^N$,
or if we want, to the function $u$ being $1$-Lipschitz
with respect to the Jacobi-Maupertuis distance $\phi_h$.
In \cite{MaVe}, Theorem 2.11 we have proved that there are constants $\alpha,\beta>0$, depending only on the masses and on the dimension of $E$ such that for any $x,y\in E^N$ and $h\geq 0$ we have 
\begin{equation} \label{estim-phi-h-holder}
\phi_h(x,y)\leq \;
\left(\alpha\norm{x-y}+h\;\beta\norm{x-y}^2\right)^{1/2}\;.
\end{equation}
It follows that each dominated function is uniformly continuous and locally H\"older continuous of degree $1/2$ with 
respect to the Euclidean distance in $E^N$.  
Again in \cite{MaVe}, Propositions 2.4 and 2.8 and Lemma 2.5  we also proved that dominated functions are exactly  viscosity subsolutions of the Hamilton-Jacobi equation $H(x,d_xu)=h$, and that every viscosity subsolution is locally lipschitz on $\Omega$.  
\begin{definition} \label{def:calibrantes}
Given $h>0$ and $u:E^N\to\R$ a function dominated by $L+h$ (or equivalently a viscosity subsolution of the
Hamilton-Jacobi equation $H(x,d_x u)=h$), 
an absolutely continuous curve $\gamma:I\to E^N$,
defined on an interval $I$,
is said to be a calibrating curve of $u$
if for any $[\alpha,\beta]\subset I$ it holds
\[
u(\gamma(\beta))-u(\gamma(\alpha))=
\calA_{L+h}(\gamma|_{[\alpha,\beta]})\,.
\]
\end{definition} 

\begin{remark}
\label{rem:calibrat-implique-geod-min}
It follows straightforwardly from the domination condition that if $\gamma$ is calibrating for $u$, then for each $[\alpha,\beta]\subset I$ 
the following identity holds
\[
\phi_h(\gamma(\alpha),\gamma(\beta))=\calA_{L+h}(\gamma\left|_{[\alpha,\beta]}\right.) .
\]
In particular, $\gamma$ is a minimizing geodesic for the Jacobi-Maupertuis distance and it is a 
genuine solution of the $N$-body problem
with energy $h$ in $\interior{I}$. 
\end{remark}

We are now ready to give the following fundamental definition.
\begin{definition}[Hyperbolic viscosity solutions]
\label{def.hyp.visc.sol}
We say that a function $u\in C^0(E^N)$
is a hyperbolic viscosity solution of $H(x,d_xu)=h$
if it satisfies the following conditions
\begin{enumerate}
    \item[(i)] The function $u$ is dominated by $L+h$.
    \item[(ii)] For any configuration $x\in E^N$ there is
    at least a curve $\gamma_x:(-\infty,0]\to E^N$
    such that  $\gamma_x(0)=x$, and such that
    $\gamma_x$ calibrates $u$. 
    \item[(iii)] All of these calibrating curves,
    up to time reversing, are hyperbolic motions with the same limit shape.
\end{enumerate}
\end{definition}

\begin{remark}
Conditions (i) and (ii) are equivalent to asking $u$
to be a fixed point of the Lax-Oleinik semigroup,
which in turn implies that the function is a viscosity solution of $H(x,d_xu)=h$.
In fact, it is enough that for each configuration $x\in E^N$
there exists a small calibrating curve of $u$  that reaches $x$
and defined on some interval $(\epsilon_x,0]$, see \cite{MaVe}.
On the other hand, in the context of Tonelli Hamiltonians without singularities, 
viscosity solutions are precisely the fixed points
of the Lax-Oleinik semigroup.
When the considered energy level is the critical one,
the calibrating curves are used to relate viscosity solutions
with the Aubry-Mather theory, which is the basis
of the weak KAM theory \cite{Fa}, \cite{FaMa}.
As far as we know, in our case we cannot assure 
that every viscosity solution must satisfy condition (ii)
since, due to the singularities of the Hamiltonian,
the argument used in \cite{FaMa} for a nonsingular Tonelli Hamiltonian  cannot be applied here.
\end{remark}

\subsection{Uniqueness of the hyperbolic solution}
\label{ss.uniqueness.hyp.sol}

The main theorem in this section ensures that, up to a constant, there is
only one hyperbolic solution to the Hamilton-Jacobi equation
$H(x,d_xu)=h$ when the limit shape is given.

\begin{definition}\label{Sa}
Given a configuration  $a=(a_1,\dots,a_N)\in\Omega$,
let us denote by $\calS_a$
the set of hyperbolic viscosity solutions
of the Hamilton-Jacobi equation $H(x,d_xu)=h$,
where $h=\frac{1}{2}\norm{a}^2$, such that  
all its calibrating curves have the configuration $a$
as the common limit shape.
\end{definition}

\begin{remark}
If $a\in\Omega$ is the limit shape
of a hyperbolic motion, then $\norm{a}=\sqrt{2h}$,
where $h>0$ is the value of the energy constant of the motion.
This explains the choice of the constant $h$
in Definition \ref{Sa}, since as we have already noticed,
any calibrating curve of a hyperbolic viscosity solution $u\in\calS_a$
must have energy $h$.
\end{remark}

In \cite{MaVe}, section 3.2 we have introduced the notion of directed horofunction. 
Let us recall that $u\in C(E^N,\R)$ is said to be a horofunction directed by the configuration $a$ if there exists a sequence $(p_n)_n$ in $E^N$ and 
a sequence $(\lambda_n)_n$ in $[0,+\infty)$ such that $\lambda_n\to +\infty$ and $p_n=\lambda_n a +o(\lambda_n)$ as $n\to +\infty$ and 
moreover for each $x\in E^N$ we have
\[
u(x)=\lim_{n\to+\infty} \phi_h(x,p_n)-\phi_h(0,p_n).
\]
We denote $\calB_h(a)$ the set of horofunctions directed by $a$. 
\begin{remark} \label{rk:inclus-horo-hyp}
It is clear from the definition that each horofunction directed by $a$ 
is dominated by $L+h$, therefore we can restate Theorems 3.2 and 3.4 in \cite{MaVe} by saying that $\calB_h(a)\subset \calS_a$, and as remarked in \cite{MaVe}, 
p. 527, we know that $\calB_h(a)\neq \emptyset$, thus $\calS_a$ is still nonempty.     
\end{remark}
We can now state the main result of this section, which is the fundamental step in the proof of Theorems \ref{theorem:1.uniqueness.Busemann}. 
The proof of this Theorem will be done in section \ref{s: proofs}.
\begin{theorem}\label{thm:main}
If $u_1$ and $u_2$ are in $\calS_a$
then $u_2=u_1+k$ for some $k\in\R$.
\end{theorem}

\begin{remark} (Calibrating curves as geodesic rays). \label{rem:cal-crv-geod-rays}
Given a hyperbolic viscosity solution $u\in\calS_a$, by Remark \ref{rem:calibrat-implique-geod-min} we know that
any calibrating curve $\gamma$ of $u$ is minimizing for $\calA_{L+h}$ on each segment contained in $(-\infty,0]$.
Since the Lagrangian satisfies $L(x,-v)=L(x,v)$ for all $x,v\in E^N$, the same minimization property is true for the curve $\gamma(-t)$ obtained by reversing the time.
Consequently, the reversed curve is a hyperbolic geodesic ray.
\end{remark}

\section{Proof of the main theorems and some consequences} \label{s: proofs}

In this section, we will prove the main theorems of this article and deduce some interesting consequences. 

\subsection{Theorem of the cone}

The proof of Theorem \ref{thm:main} makes a strong use of
the following basic result, which will be proved
in sections \ref{sec:uniq.hyp.motions} and \ref{sec:proof.thm.cone}.
Furthermore, it leads to the announced differentiability
of Busemann functions over certain regions.

In order to state it, we need to introduce some notation.
For $\alpha\in (0,1)$ and $r>0$, let us set
\[
\cone_a(\alpha)=
\set{x\in E^N \mid\;\inner{x}{a} \geq  \alpha\,\norm{x}\norm{a}}
\qquad \text{and}\qquad 
\cone_a(\alpha,r)=\cone_a(\alpha)\setminus B(r)
\]
where $B(r)$ is as usual the open ball centered at the origin
with radius $r>0$.

The set $\cone_a(\alpha)$ is  the cone in $E^N$ with axis $\R^+a$
and angle equal to $\arccos \alpha$,
while  $\cone_a(\alpha,r)$ will be called \emph{cutted cone}.
From now on we set
\[
\alpha_0(a)=
\inf\set{\alpha\in (0,1)\mid \cone_a(\alpha)\cap\Delta=\set{0}}.
\]

\begin{theorem} [Theorem of the cone]\label{thm:cone}
Let $a\in\Omega$ be a given configuration, and $\alpha\in (0,1)$ such that
$\alpha>\alpha_0(a)$.
There is $r>0$ such that for any $x\in \cone_a(\alpha,r)$
the following two assertions hold.

\vspace{.2cm}
(1) There exists a unique hyperbolic motion
$\gamma:[0,+\infty)\to\Omega$
with limit shape $a$, starting at $x$, and such that
$\gamma(t)\in \cone_a(\alpha,r)$ for all $t\ge 0$.

\vspace{.2cm}
(2) The curve
  $\gamma$ is the unique hyperbolic geodesic ray starting at $x$ and having limit shape $a$.  
\end{theorem}

\begin{remark} \label{rem:kepler}
For the planar Kepler problem, given a non-zero vector $a$ of the plane, and any initial position $x\notin \R_+a$, we always have exactly two hyperbolic motions starting 
at $x$ and with limit shape given by $a$. Indeed, denting $P_+$ (resp. $P_-$) the half plane that is located at the right (resp. left) of the oriented line generated by $a$, let us consider the family of all keplerian hyperbolas with energy $h=\norm{a}^2/2$, positive angular momentum and axis given by 
$\R_+a$. We can now rotate each keplerian hyperbola so that its limit shape in the future becomes equal to $a$. It is clear now that each initial position $x\notin \R_+a$ lies in exactly one hyperbola of this new family, thus we can consider the branch of hyperbola starting from $x$.   If $x\in P_+$, the total variation of the polar 
angle along this motion is less or equal to $\pi$, and we will say that this branch of hyperbola is direct, while if  $x\in P_-$ this total variation is bigger or equal to $\pi$, and we will say that this branch of hyperbola is indirect. By repeating the same construction with the family of keplerian hyperbolas having negative angular momentum, 
we see that for each $x\in \R^2\setminus \R_+a$ there are exactly two branches of keplerian hyperbolas starting at $x$ and with limit shape $a$, one is  direct and the other is indirect.  If $x\in -\R_+a$, both branches are half-hyperbolas, they are symmetric with respect to the line $\R a$ and are at the same time direct and indirect.  An illustration of these two branches of hyperbolas can be find in Figure 5, pp. 545 in \cite{MaVe}.  
It is clear that an indirect branch is never contained in a cone $\cone_a(\alpha)$. Moreover, by using Lambert's Theorem (see for instance \cite{Alb}), one can prove that when $x\notin \R a$, the indirect branch is never a geodesic ray. This exemple shows that the uniqueness of the hyperbolic motion $\gamma$ in Theorem \ref{thm:cone} is no more true 
if we do not require that $\gamma(t)$ is all the time contained in the cone $\cone_a(\alpha,r)$ or that $\gamma$ is a geodesic ray. 
\end{remark}

\begin{corollary}[Differentiability of hyperbolic solutions]
\label{coro:diff.ty.in.cone}
Given $a\in\Omega$ there is a cutted cone $\cone=\cone_a(\alpha,r)$
such that any $u\in\calS_a$ is differentiable on $\cone$.
\end{corollary}
\begin{proof}
Here we use the characterization of the points of differentiability
given by Proposition \ref{prop:diff-unic-calibr} in the Appendix \ref{ViscoSol}.
The differentiability of a hyperbolic viscosity solution $u\in\calS_a$
in a configuration without collisions $x\in\Omega$ is equivalent
to the uniqueness of the calibrating curve $\gamma_x:(-\infty,0]\to E^N$
reaching $x$ at $t=0$.

Now, given $a\in\Omega$ and $u\in\calS_a$, let $\cone=\cone_a(\alpha,r)$
be any cutted cone given by Theorem \ref{thm:cone}.
Let $x\in\cone$ and suppose that $\gamma_x,\gamma_x':(-\infty,0]\to E^N$
are two calibrating curves of $u$, reaching $x$ at $t=0$.
As shown in Remark \ref{rem:cal-crv-geod-rays}, the reversed curves 
defined on $[0,+\infty)$ by $\gamma(t)=\gamma_x(-t)$
and $\gamma'(t)=\gamma_x'(-t)$ are both geodesic rays starting at $x\in\cone$,
and both having limit shape $a$. By assertion (2) of Theorem \ref{thm:cone}
we have $\gamma=\gamma'$, hence that $\gamma_x=\gamma_x'$.
Therefore $u$ is differentiable at any $x\in\cone$.
\end{proof}

\subsection{Proof of the main results}

\label{ss:proofmain}

In order to prove the theorems stated in subsections  \ref{ss:Hyp-mot-bsm-fnct} and \ref{ss.uniqueness.hyp.sol}, 
let us fix some notations.
For $h>0$ we define the set
\[
\Omega_h=\set{a\in\Omega\mid \norm{a}=\sqrt{2h}}.
\]
If $a\in\Omega_h$ we denote $\calH_a$ the set of absolutely continuous curves $\gamma:[t_0,+\infty)\to E^N$ such that $\gamma\left|_{(t_0,+\infty)}\right.$
is a hyperbolic motion with limit shape $a$, and we denote $\calM_a$ the set of those $\gamma\in\calH_a$ that are a hyperbolic geodesic ray at energy level $h$.
\begin{proof}[Proof of Theorem \ref{thm:main}]
Let $\alpha\in (\alpha_0(a),1)$ be fixed, and
$r>0$ given by the theorem of the cone (Theorem \ref{thm:cone}).
By choice of $\alpha$
we know that $\cone_a(\alpha,r)\subset \Omega$.
As already noticed in Remark \ref{rem:cal-crv-geod-rays}, each calibrating curve of a hyperbolic viscosity solution 
is - up to time reversing - a hyperbolic geodesic ray. Moreover, by (2) of Theorem \ref{thm:cone}, if $x\in\cone_a(\alpha,r)$, 
there is a unique element of $\calM_a$ starting from $x$ at time $t=0$, therefore 
if $u_1$ and $u_2$ are in $\calS_a$ and $x\in \cone_a(\alpha,r)$, 
we know that each global calibrating curve for $u_1$ or $u_2$
arriving at $x$ must agree - up to time reversing -
with the unique element of $\calM_a$ starting at $x$ at time $t=0$.
Since the differentiability of a hyperbolic viscosity solution at a configuration without collisions 
is characterized by the uniqueness of the calibrating curve
(see Proposition \ref{prop:diff-unic-calibr} in Appendix \ref{ViscoSol}),
we have that $u_1$ and $u_2$ are differentiable at configuration $x$,
and moreover, $d_xu_1=d_xu_2$.
In particular $u_2-u_1$ is constant on the cutted cone $\cone_a(\alpha,r)$.
Denote $k=u_2(x)-u_1(x)$ for $x\in \cone_a(\alpha,r)$.
Let us show now that $u_2-u_1$ is constant on the whole configuration space $E^N$.
Given $x\in E^N$, by definition of hyperbolic viscosity solution, there exists $\gamma\in \calH_a$, starting at $x$ at time $t=0$,
such that the curve obtained by reversing the time is calibrating for $u_1$. As already remarked, the calibration property implies that $\gamma\in\calM_a$.
Since $\gamma(t)=ta+o(t)$ as $t\to +\infty$, then
\[
\inner{\gamma(t)}{a}=t\norm{a}^2+o(t),\quad \text{and} \quad \norm{\gamma(t)}=t\norm{a}+o(t),\quad \text{as}\quad t\to +\infty,
\]
therefore, by definition of $\cone_a(\alpha,r)$,
there exists $t>0$ such that $\gamma(t)\in \cone_a(\alpha,r)$.
By the calibrating identity for $u_1$ and the domination inequality
for $u_2$ we have   
\[
\begin{array}{rl}
u_1(x)-u_1(\gamma(t))&=\calA_{L+h}(\gamma\left|_{[0,t]}\right.) \\
u_2(x)-u_2(\gamma(t))&\leq \calA_{L+h}(\gamma\left|_{[0,t]}\right.),
\end{array}
\]
therefore $u_2(x)-u_1(x)\leq k$.
Repeating the argument by exchanging $u_1$ and $u_2$
we find that $u_2(x)-u_1(x)=k$ for all $x\in E^N$.
\end{proof}
 Theorems \ref{theorem:1.uniqueness.Busemann} and \ref{thm:diff-busemann} are now direct consequences of Theorems \ref{thm:main} and \ref{thm:cone} and of the following 
\begin{lemma} \label{lem:busemann-vs-hypviscsol}
Given $h>0$, $a\in\Omega_h$ and $\gamma\in \calM_a$, supposed defined on $[0,+\infty)$ then 
\begin{itemize}
    \item[i)] $b_\gamma-b_\gamma(0)$ is a horofunction directed by $a$. 
    \item[ii)] $b_\gamma$ is a hyperbolic viscosity solution with limit shape given by $a$.
    \item[iii)] The curve $\sigma : (-\infty,0]\to E^N$, defined by $\sigma(t)=\gamma(-t)$ is calibrating for $b_\gamma$.
\end{itemize}
\end{lemma}
\begin{proof}
By definition of Busemann function, for each $x\in E^N$ we have
\[
b_\gamma(x)-b_\gamma(0)=\lim_{n\to +\infty} \phi_h(x,\gamma(n))-\phi_h(0,\gamma(n)).
\]
Since $\gamma$ is a hyperbolic motion with limit shape $a$ it holds
\[
\gamma(n)=na+o(n),\quad n\to +\infty.
\]
This observation shows that $b_\gamma-b_\gamma(0)$ is a horofunction directed by $a$, therefore i) is proved. In  
Remark \ref{rk:inclus-horo-hyp}  we have already noticed that each horofunction directed by $a$ is a hyperbolic viscosity solution, 
and it follows from the definition of hyperbolic viscosity solution  that $\calS_a$ is invariant by addition with a constant, that is 
$b_\gamma\in \calS_a$, as stated in ii). Since $\gamma$ is a geodesic ray, for any $t\geq 0$ we have 
\[
\begin{array}{rl}
b_\gamma(\gamma(0))-b_\gamma(\gamma(t))&=\lim_{s\to +\infty}  \phi_h(\gamma(t),\gamma(s))-\phi_h(\gamma(0),\gamma(s)) \\
&=\phi_h(\gamma(0),\gamma(t))=\calA_{L+h}(\gamma\left|_{[0,t]}\right.).
\end{array}
\]
This identity can be expressed by saying that the reversed curve $\sigma$ is calibrating for $b_\gamma$, as stated in iii).
\end{proof}
\begin{remark} \label{rem:sol-hyp-vs-busemann}
Given $a\in\Omega$ and $u\in\calS_a$, if $p\in E^N$ then $u-u(p)$ is a Busemann function associated to a hyperbolic geodesic ray with limit shape given by $a$. Indeed, let $\gamma : [0,+\infty)\to E^N$ be a hyperbolic geodesic ray 
 starting at the point $p$ at $t=0$ and obtained by reversing a curve that calibrates $u$. By Lemma \ref{lem:busemann-vs-hypviscsol} we know that $b_\gamma\in\calS_a$. Moreover, both $b_\gamma$ and $u-u(p)$ vanish at $p$. Since $u-u(p)\in\calS_a$, by Theorem \ref{thm:main} we have that 
 $b_\gamma=u-u(p)$. We do not know if $u\in\calS_a$ is always a Busemann function, or equivalently, if every  hyperbolic viscosity solution vanishes at some point 
 $p\in E^N$.
 
\end{remark}

\begin{proof}[Proof of Theorem \ref{theorem:1.uniqueness.Busemann}]
Let $h>0$ and $\gamma_1$, $\gamma_2$ be two hyperbolic geodesic rays with their respective limit shape $a_1,a_2\in\Omega$. 
Without loss of generality we can assume that $\gamma_1$ and $\gamma_2$ are both defined on $[0,+\infty)$.  
The fact the $(1)$ and $(3)$ are equivalent is a direct consequence of
Chazy's estimate, that we have recalled in Remark \ref{remark:Chazy}.
On the other hand, it is not difficult to see that $(1)$ and $(2)$
are also equivalent. Indeed, in \cite{MaVe} Lemma 4.4, we have proved that for any $x,y\in E^N$ we have 
\[
\phi_h(x,y)\geq \sqrt{2h}\norm{x-y},
\]
therefore, if $\phi_h(\gamma_1(t),\gamma_2(t))$ is bounded, then $\norm{\gamma_1(t)-\gamma_2(t)}$ is bounded too.
Conversely, if we assume that $\norm{\gamma_1(t)-\gamma_2(t)}$ is bounded from above by a constant $A$,  we know from what we have proved that $a_1=a_2$, thus we set $a=a_1=a_2$.  Let now $\alpha\in (\alpha_0(a),1)$ and $r>0$. Since $\gamma_1$ and $\gamma_2$ are both hyperbolic motions with limit shape $a$, we know that there exists $T>0$ such that $\gamma_i(t)\in \cone_a(\alpha,r/\alpha)$ for $i=1,2$ and for all $t\geq T$. Let $d_t=(2h)^{-1/2}\norm{\gamma_1(t)-\gamma_2(t)}$ and 
$\delta_t: [0,d_t]\to E^N$ be the segment joining $\gamma_1(t)$ with $\gamma_2(t)$, parametrized with constant speed equal to $\sqrt{2h}$. 
It is immediate to verify that 
if $t\geq T$ and $s\in [0,d_t]$ then $\delta_t(s)\in \cone_a(\alpha,r)$. Since $U$ is positive, continuous on $\Omega$ and homogeneous of degree $-1$, by choice of $\alpha$ there exists $\mu>0$ such that for each $x\in \cone_a(\alpha,r)$ we have $U(x)\leq \mu/\norm{x}$, therefore if $t\geq T$ it holds
\[
\phi_h(\gamma_1(t),\gamma_2(t))\leq \calA_{L+h}(\delta_t)\leq \left( 2h+\frac{\mu}{r}\right)d_t\leq  \sqrt{2h}\left(1+\frac{\mu}{2hr}\right)A,
\]
in particular $\phi_h(\gamma_1(t),\gamma_2(t))$ is bounded. 

Now, we only have to prove the equivalence between $(3)$ and $(4)$.
 Let us denote $b_{\gamma_1}$ and $b_{\gamma_2}$ Busemann functions associated with $\gamma_1$ and $\gamma_2$. If $a_1=a_2$, by Lemma \ref{lem:busemann-vs-hypviscsol} and Theorem \ref{thm:main} we have that $b_{\gamma_1}-b_{\gamma_2}$ is constant. Conversely, let us suppose that $b_{\gamma_1}-b_{\gamma_2}$ is constant. If we set $h_i=\norm{a_i}^2/2$, $i=1,2$ and we denote $\sigma_i : (-\infty,0]\to E^N$, $\sigma_i(t)=\gamma_i(-t)$ the reversed curve of $\gamma_i$, by iii) of Lemma \ref{lem:busemann-vs-hypviscsol} we know that 
$\sigma_i$ is calibrating for $b_{\gamma_i}$ at level $h_i$.  
Applying  again item ii) of Lemma \ref{lem:busemann-vs-hypviscsol} we know that $b_{\gamma_1}$ is a hyperbolic viscosity solution with limit shape 
given by $a_1$. Since  $\calS_{a_1}$ is invariant by addition with a constant, we have 
$b_{\gamma_2}\in\calS_{a_1}$. In particular, all calibrating curve of $b_{\gamma_2}$ - up to time reversing - are hyperbolic motions with limit shape $a_1$. 
Since $\gamma_2$ is the reversed curve of $\sigma_2$, which is calibrating for $b_{\gamma_2}$, the limit shape of $\gamma_2$ is $a_1$, that is $a_1=a_2$.  
\end{proof} 

\begin{proof}[Proof of Theorem \ref{thm:diff-busemann}]
Let $\gamma:[t_0,+\infty)\to\Omega$ be any hyperbolic motion with limit shape $a\in\Omega$, and let $h=\norm{a}^2/2$ be 
the energy constant of $\gamma$. 
Now we choose a cutted cone $\cone_a(\alpha,r)$
around the configuration $a$ for which
the conclusion of the theorem of the cone (Theorem \ref{thm:cone}) holds.
Since we have that $\gamma(t)=ta+o(t)$ as $t\to +\infty$,
we deduce that there exists $t_1\geq t_0$ such that
$\gamma(t)\in\cone_a(\alpha,r)$ for all $t\geq t_1$.
Let us denote by $x_1$ the configuration
$\gamma(t_1)\in\cone_a(\alpha,r)$.
Then by the theorem of the cone, the restriction of $\gamma$
to $[t_1,+\infty)$ is in the set of hyperbolic geodesic rays $\calM_a$. By Lemma \ref{lem:busemann-vs-hypviscsol} and the invariance of $\calS_a$ by addition with a constant we know that $b_a$ is a hyperbolic viscosity solution, therefore 
for each $x\in E^N$ there is a calibrating curve for $b_a$ arriving at $x$ which is the time reversion of a hyperbolic geodesic ray with limit shape $a$.
By the theorem of the cone, for each $x\in \cone_a(\alpha,r)$ there is a unique element of $\calM_a$ starting at $x$, therefore we can say that 
there is a unique calibrating curve arriving at $x$. 
In particular, up to time translation, for each $t\geq t_1$ the curve obtained by reversing  $\gamma\left|_{[t,+\infty)}\right.$ is the unique calibrating curve for $b_a$ arriving at $\gamma(t)$. 
By Proposition \ref{prop:diff-unic-calibr} in the Appendix \ref{ViscoSol} we know now that $b_a$ is differentiable in $\cone_a(\alpha,r)$,
and its gradient at each point of the cone is equal
to the velocity of the unique calibrating curve that passes through it.
Then for $t\geq t_1$ we have $\dot\gamma(t)=-\nabla b_a(\gamma(t))$.
\end{proof}

\begin{remark}
By Theorem \ref{thm:diff-busemann} and Lemma \ref{lem:busemann-vs-hypviscsol} we have that if $\gamma : [t_0,+\infty)\to \Omega$ is a hyperbolic motion, then 
there exists $t_1\geq t_0$ such that $\gamma^\prime=\gamma\left|_{[t_1,+\infty)}\right.$ is a hyperbolic geodesic rays, and in particular, the curve obtained 
by reversing $\gamma^\prime$ is calibrating for $b_{\gamma^\prime}$.
\end{remark}

\subsection{Corollaries of the main theorems}
\label{ss:corollaries}

Several corollaries follow from the main theorems.
In order to state them we need to introduce some notions 
as well as to recall some facts.
In our previous work \cite{MaVe} we have defined,
for $h>0$, the 
Gromov's ideal boundary of the length space $(E^N,\phi_h)$, here denoted as $\calB_h$. If $a\in\Omega_h$ then 
the set $\calB_h(a)$ of horofunctions directed by $a$ is included in $\calB_h$.
Let us recall this construction.
Given $h>0$, let us call $\calD_h$
the set of functions dominated by $L+h$ and vanishing at $0\in E^N$
\[
\calD_h=\set{u:E^N\to \R\mid u(0)=0, \text{ and }
u(y)-u(x) \leq \phi_h(x,y) \text{ for all }x,y\in E^N}.
\]
By inequality (\ref{estim-phi-h-holder}) we have that $\calD_h$ is an equicontinuous and equibounded family of functions, 
or equivalently, it is a compact subset of $C^0(E^N,\R)$ for the topology of uniform convergence
on compact sets. This is one of the main tools we used in \cite{MaVe}
to construct hyperbolic viscosity solutions as directed horofunctions.
For each $p\in E^N$, let us denote 
\[
u_p : E^N\to \R,\qquad u_p(x)=\phi_h(x,p)-\phi_h(0,p).
\]
It follows straightforwardly that $u_p\in\calD_h$. In particular, for any sequence $(p_n)_n$ in $E^N$,
we can always extract
a subsequence of ($u_{p_n})_n$ that converges uniformly on every compact subset of $E^N$.
\begin{definition}[Gromov's Ideal boundary]
For any given $h\geq 0$, the set of functions $u\in C^0(E^N)$
which are the limit of a sequence
$u_{p_n}(x)=\phi_h(x,p_n)-\phi_h(p_n,0)$
where $\norm{p_n}\to +\infty$ is called the Gromov
ideal boundary of level $h$, and it is denoted $\calB_h$.
\end{definition}
The elements of the ideal boundary $\calB_h$
are called horofunctions.
Clearly we have that $\calB_h\subset\calD_h$, and if $a\in\Omega_h$ then $\calB_h(a)\subset \calB_h$, where $\calB_h(a)$ 
denotes the set of horofunctions directed by $a$, that we have introduced in subsection \ref{ss.uniqueness.hyp.sol}.
We have proved in our previous work that each
function in $\calB_h$ is a global viscosity solution to the
Hamilton-Jacobi $H(x,d_xu)=h$,
see \cite[Thm. 3.1 p. 524]{MaVe}. Actually, we proved
that they are fixed points of the Lax-Oleinik semigroup,
that is to say, they verify conditions
(i) and (ii) in the definition of hyperbolic viscosity solution
(see Definition \ref{def.hyp.visc.sol}). 
The corollaries that we will deduce from Theorems \ref{theorem:1.uniqueness.Busemann}, \ref{thm:diff-busemann} and \ref{thm:main}
concern the hyperbolic geodesic rays, their associated Busemann functions,
and the Gromov ideal boundary for the Jacobi-Maupertuis distance
of positive energy levels. We can now state and prove them.

\begin{corollary}[Uniqueness of the directed horofunction]
\label{coro.unique.direct.horo}
Let $h>0$ and $a\in\Omega_h$. Then the Gromov's ideal boundary
of the Jacobi-Maupertuis metric $j_h$ contains
a unique horofunction directed by the configuration $a$. This unique directed horofunction is exactly the Busemann function $b_a$ of the configuration $a$.
\end{corollary}
\begin{proof}
Since $a\in\Omega_h$, by Remark \ref{rk:inclus-horo-hyp} we have that $\calB_h(a)\neq\emptyset$,
and $\calB_h(a)\subset\calS_a$. Since any horofunction
$u\in\calB_h$ satisfy $u(0)=0$, we deduce from Theorem \ref{thm:main}
that $\calB_h(a)$ is actually a singleton. Let now $\gamma\in\calM_a$ and $b_\gamma$ be the associated Busemann function. 
By Lemma \ref{lem:busemann-vs-hypviscsol} we know that $b_a=b_\gamma-b_\gamma(0)$ is exactly the horofunction directed by $a$, therefore $\calB_h(a)=\{b_a\}$.
\end{proof}

Finally, we also deduce the following corollary,
whose geometric interpretation is that the \emph{cut locus}
of a configuration $a\in\Omega$ (seen as a boundary point at infinity)
has zero measure.

\begin{corollary}[generic uniqueness of the geodesic ray]
\label{coro.generic.unicity}
Given $a\in\Omega$ we have that, for almost all $x\in E^N$
there is a unique hyperbolic geodesic ray $\gamma\in\calM_a$
such that $\gamma(0)=x$.
\end{corollary}
\begin{proof}
We use Rademacher's theorem applied to the Busemann function $b_a$
over the open and dense set $\Omega$ of configurations without collisions.
Indeed, we know that $b_a$ (as well as any dominated function)
is differentiable almost everywhere in $\Omega$,
since it is locally Lipschitz (see \cite[Lemma 2.5]{MaVe}).
Actually the local Lipschitz constant diverges as the configurations
approach the set $\Delta=E^N\setminus\Omega$.

Let $x\in\Omega$ be a configuration for which two different
geodesic rays $\gamma$ and $\gamma'$ start, both with configuration
$a$ as limit shape. Let us call $\sigma$ and $\sigma'$ the curves
obtained by reversing the time.
That is, $\sigma:(-\infty,0]\to\Omega$ is defined by $\sigma(t)=\gamma(-t)$,
and $\sigma'$ is defined similarly. By Lemma \ref{lem:busemann-vs-hypviscsol} we know that 
$\sigma$ is calibrating for the Busemann function defined by $\gamma$, namely
de function $b_\gamma$. By the same reason, the curve $\sigma'$
is calibrating for the Busemann function $b_{\gamma'}$.
By Theorem \ref{theorem:1.uniqueness.Busemann} and by definition of $b_a$ have that $b_a=b_\gamma-b_\gamma(0)=b_{\gamma^\prime}-b_{\gamma^\prime}(0)$, therefore 
both $\sigma$ and $\sigma'$ are calibrating for the function $b_a$.
By Proposition \ref{prop:diff-unic-calibr} we conclude that $b_a$ is not
differentiable at $x$, and we know that the set of these points has zero
measure in $\Omega$ by Rademacher's theorem.
\end{proof}

\begin{remark}
If we drop the minimizing assumption, and we only ask to $\gamma$
to be a hyperbolic motion with limit shape $a$,
the previous corollary is no more true.
This already occurs in the planar Kepler problem. 
Indeed, as we point out in Remark \ref{rem:kepler},
excepted for the case in which the initial position $x$
is on the half line starting from the origin and directed by $a$,
there are exactly two hyperbolic motions starting from $x$ and with limit shape $a$.
\end{remark}

\section{The limit shape map}

We recall that, for any hyperbolic motion such that
$x(t)=ta+o(t)$ as $t\to +\infty$, we have 
$a=\lim_{+\infty}(x(t)/t)=\lim_{+\infty}\dot x(t)$.
That is, the limit shape is precisely the asymptotic velocity.
The second identity is a consequence of Chazy's estimate
that we have already mentioned in Remark  \ref{remark:Chazy}. 

The equation of motion of the $N$-body problem defines a local flow on $\Omega\times E^N$, here denoted 
\[
\varphi : \calD\to \Omega\times E^N,\quad (x_0,v_0,t)\mapsto \varphi^t(x_0,v_0),
\]
where $\calD$ is an open subset of $\Omega\times E^N\times \R$.
Let us recall quickly how $\calD$ and $\varphi$ are  defined.
Equation of motion can be written as the following first order differential equation on $\Omega\times E^N$ 
\begin{equation}
\label{eqn-mouv-1er-ordre}
\left\{
\begin{array}{rcl}
\dot{x}&=&v \\
\dot{v}&=&\nabla U(x).
\end{array}
\right.
\end{equation}
Given $z_0=(x_0,v_0)\in \Omega \times E^N$, denoting $(x,v): I_{z_0} \to \Omega\times E^N$ the maximal solution of (\ref{eqn-mouv-1er-ordre}) with initial condition
$x(0)=x_0$ and $v(0)=v_0$, we define 
\[
\calD=\bigcup_{z_0\in \Omega\times E^N} \{z_0\}\times I_{z_0}
\]
and if $t\in I_{z_0}$ then we set $\varphi^t(x_0,v_0)=(x(t),v(t))$. We recall that if $(z_0,t)\in\calD$ and $(\varphi^t(z_0),s)\in\calD$ then $(z_0,t+s)\in\calD$ and moreover $\varphi^{t+s}(z_0)=\varphi^s(\varphi^t(z_0))$. 

\begin{definition}
[Hyperbolic initial conditions]
Let $\calH\subset T\Omega$ be
the set of initial conditions $(x,v)$ such that
\begin{enumerate}
    \item[(1)] The evolution $(x(t),\dot x(t)) = \varphi^t(x,v)$
    is defined for all $t\geq 0$, and
    \item[(2)] The motion $x(t)$ is hyperbolic as $t\to +\infty$.
\end{enumerate}
\end{definition}

\begin{definition}
[Limit shape map]
The limit shape map is the function
\[
\sm :\calH\to\Omega,\;\;\;
\sm(x,v)=\lim_{t\to +\infty}\,\frac{\pi(\varphi^t(x,v))}{t}
=\lim_{t\to +\infty}\,\frac{x(t)}{t}
\]
where $\pi:T\Omega\to\Omega$ is the tangent bundle map,
that is, the projection $(x,v)\mapsto x$.
\end{definition}

The fundamental result regarding this function is what we called
Chazy's lemma, which asserts that its domain of definition $\calH$
is an open subset of $T\Omega$, and $\sm$ is a continuous map. 
For more clarity, we recall here the statement of Chazy's Lemma. 
\begin{lemma}[Lemma 4.1 in \cite{MaVe}]
\label{lem:lema-cont.limitshape}
Let $U:E^N\to\R\cup\set{+\infty}$ be an homogeneous potential
of degree $-1$ of class $C^2$ on the open set
$\Omega=\set{x\in E^N\mid U(x)<+\infty}$.
Let $x:[0,+\infty)\to\Omega$ be a given solution of
$\ddot x=\nabla U(x)$ satisfying $x(t)=ta+o(t)$
as $t\to +\infty$ with $a\in\Omega$.

Then we have:
\begin{enumerate}
\item[(1)] The solution $x$ has asymptotic velocity $a$,
meaning that
\[
\lim_{t\to+\infty}\dot x(t)=a\,.
\]
\item[(2)] (Chazy's continuity of the limit shape)
Given $\epsilon>0$,
there are constants $t_1>0$ and $\delta>0$ such that,
for any maximal solution $y:[0,T)\to\Omega$ satisfying
$\norm{y(0)-x(0)}<\delta$ and
$\norm{\dot y(0)-\dot x(0)}<\delta$,
we have:
\begin{enumerate}
\item[(i)] $T=+\infty$, $\norm{y(t)-ta}<t\epsilon$ for all $t>t_1$,
and moreover
\item[(ii)] there is $b\in\Omega$ with $\norm{b-a}<\epsilon$ 
for which $y(t)=tb+o(t)$.
\end{enumerate}
\end{enumerate}
\end{lemma}

In this section we go further in understanding
the regularity properties of the limit shape map, and 
we will show that $\sm$ is of class $\calC^1$.
Moreover, we will need an integral expression
for the differential $d\sm$ in terms of Jacobi fields,
which is given in Proposition \ref{differ-fig-limite} below.
We will make use of Jacobi fields along our motions,
and this is why we recall them now.

\subsection{The Jacobi equation}
\label{ss.Jacobi.equ}

Given an interval $I\subset\R$ and a motion $\gamma: I\to\Omega$,
a \emph{Jacobi field} along $\gamma$ is a vector field $J:I\rightarrow E^N$
which satisfy the \emph{Jacobi equation}
\[
\ddot{J}=D^2U(\gamma)\,J
\]
where here $D^2U(x):E^N\rightarrow E^N$ denotes the Hessian of $U$
at a configuration $x\in\Omega$ with respect to the mass inner product,
that is to say, the unique endomorphism
$\xi\mapsto D^2U(x)\,\xi$ for which
\[
d_x^{\,2}\,U(\xi,\eta)=\inner{D^2U(x)\,\xi}{\eta}
\]
holds for any $\xi$ and $\eta$ in $E^N$.

\begin{remark}
The gradient is defined by the equation
$d_xU(\eta)=\inner{\nabla U(x)}{\eta}$
for all $\eta\in E^N$. Therefore,
seeing mass inner product as a constant
Riemannian metric we have that
$D^2U(x)\,\xi=\nabla_\xi\nabla U(x)$ for any $\xi\in E^N$.
\end{remark}
\begin{remark}
If we write Newton's equation as the first order differential equation
(\ref{eqn-mouv-1er-ordre})
in the phase space, then its linearized form along a given solution $(\gamma,\dot\gamma)$
writes $\dot J=V$ and $\dot V=\nabla_J\nabla U(\gamma)=D^2U(\gamma)J$.
Thus, the first component $J$ of the solutions of this system are precisely
the Jacobi fields along the given motion $\gamma$.
\end{remark}

However, we recall that the natural way to obtain the Jacobi fields
is seeing them as variational fields.
Let $\gamma:I\to \Omega$ be any motion,
and fix an initial time $t_0\in I$.
Let us denote $x_0=\gamma(t_0)$ and $v_0=\dot\gamma(t_0)$.
Then, for any  $(X,V)\in E^N \times E^N$ we can produce
a variation of $\gamma$ by setting
\[
F(t,s)=\pi(\varphi^{t-t_0}(x_0+sX,v_0+sV))
\]
for the values of t and s at which the expression is well defined.
Then
\[
J(t)=\frac{\partial F}{\partial s}(t,0)
=\pi(d_{(x_0,v_0)}\varphi^{t-t_0}(X,V))
\]
is the solution to the Jacobi equation along $\gamma$
satisfying $J(t_0)=X$ and $\dot J(t_0)=V$.

\subsection{A lemma on the growth of Jacobi fields}
\label{ss:growth-jacobi}
 
Here we prove a preliminary result on the growth of families
of Jacobi fields along hyperbolic motions.

We start with a lemma on linear second order differential equations
on the real line.
Although we are convinced that this elementary lemma
should be familiar to experts in linear differential equations
whose coefficients tend to constants, we have not found any reference
in the literature that expresses it.
For this reason we give here the proof that we have found.
 
\begin{lemma} \label{croissance_eq_lin}
Let $x:[1,+\infty)\to \R^+$ be a positive function such that
\[
\ddot x(t)=\,k\,{t^{-(2+\alpha)}}\,x(t)
\]
for all $t\geq 1$, where $\alpha$ and $k$ are positive constants.
Then, there is $c>0$ such that
\[
x(t)\leq c\,t
\]
for all $t\geq 1$.
\end{lemma}

\begin{remark}
Note that the lemma is false if we allow the constant $\alpha$ to be zero.
\end{remark}

\begin{proof}
Let us fix $0 <\delta<\alpha$,
and define an auxiliary function on the same interval by
\[
\varphi(t)=B\,t^{1+\delta}.
\]
We will prove that for some choice of $B>0$
we have $x(t)\leq\varphi(t)$ for all $t\geq 1$.
This is enough to complete the proof, since by integrating both members in the inequality
\[
\ddot{x}(t)=kt^{-(2+\alpha)}x(t)\leq kBt^{-(1+\alpha-\delta)}
\]
we get the following uniform bound 
\[
\dot{x}(t)\leq \dot{x}(1)+\frac{k B}{\alpha-\delta}, \qquad t\geq 1,
\]
therefore $x(t)\leq ct$ for some  $c>0$ which is constant with respect to $t$.
Let us prove now that $x(t)\leq \varphi(t)$ for all $t\geq 1$.
If $t\geq 1$ is such that $x(s)\leq\varphi(s)$ for any $s\in [1,t]$
then we can write
\[
\dot x (t)\leq \dot x_1 +\;
\int_1^t \,k\,{s^{-(2+\alpha)}}\varphi(s)\,ds
\]
where $\dot x_1=\dot x(1)$, and as before we have
\[\dot x(t)\leq\,
\dot x_1 +\,\int_1^{+\infty}\,
\frac{kB}{s^{1+\alpha-\delta}}\,ds 
= \dot x_1 + \frac{kB}{\alpha-\delta}\,.
\]

We want now to compare the derivative of our function $x$
with the derivative of the auxiliary function $\varphi$.
By the previous estimate,
assuming again that $x(s)\leq\varphi(s)$ for any $s\in [1,t]$,
we know that a sufficient condition to have
$\dot x(t)< \dot\varphi(t)$
is given by the inequality
\begin{equation}\label{condition-lemma-nov}
\dot x_1 <
B\left(\,(1+\delta)\,t^\delta-\frac{k}{\alpha-\delta}\,\right).
\end{equation}
Then the argument can be concluded as follows.
We take $T>1$ large enough such that
\[
(1+\delta)\,T^\delta>\frac{k}{\alpha-\delta}
\]
and choose $B>0$ big enough
so that condition (\ref{condition-lemma-nov})
is verified for all $t\geq T$ and moreover $x(t)\leq \varphi(t)= B t^{1+\delta}$ for all $t\in [1,T]$.
We claim that for these choices of $T$ and $B$,
we have $x(t)\leq\varphi(t)$ for all $t\geq 1$.
Otherwise there should be some $t>T$ such that
$x(s)\leq\varphi(s)$ for all $s\in[1,t]$, with equality for $s=t$, and 
moreover $\dot x(t)<\dot\varphi(t)$, which is not possible.
\end{proof}

Now, we use this lemma to prove a uniform upper bound
for a family of Jacobi fields along hyperbolic motions.

\begin{lemma}
\label{croiss-jacobi}
Let $\calF=\set{\gamma_s\mid s\in A}$
be a family of motions of the $N$-body problem, all defined for $t\geq 0$,
and assume there is a compact set $K\subset\Omega$, and a constant $b>0$
such that for any $s\in A$ and any $t\geq 0$ we have:
\begin{enumerate}
\item[(i)]
$u_s(t)=\gamma_s(t)/\norm{\gamma_s(t)}\in K$, and
\item[(ii)]
$\norm{\gamma_s(t)}\geq b\,(t+1)$.
\end{enumerate}
Then, there is a constant $c>0$ such that 
\[
\norm{J(t)}\leq c\,(t+1)
\]
for all $t\geq 0$, where $J$ is any Jacobi field along
any $\gamma_s$ such that $J(0)$ and $\dot J(0)$ are
both in the unit ball of $E^N$.
\end{lemma} 

\begin{proof}
Let us start by defining the constant,
En reponse à la remarque {\it Further comments about presentation, 18} du referee. C'est à propos de la notation pour la Hessienne de $U$.
\begin{equation} \label{def-const-k}
k=b^{-3}\max\set{\,\norm{D^2U(y)}\;\mid y\in K},
\end{equation}
and observing that, since the Hessian operator $D^2U$
is homogeneous of degree $-3$, it follows from (i) and (ii) that 
\begin{equation} \label{borne-HU-hom}
\norm{D^2U(\gamma_s(t))}=
\norm{D^2U(u_s(t))}\,\norm{\gamma_s(t)}^{-3}\leq k\,(t+1)^{-3},
\end{equation}
for any $s\in A$ and $t\geq 0$.
We now make a claim that will allow us to compare
the norm of a Jacobi field with the size of a solution
of a one-dimensional equation.

\vspace{.2cm}
\noindent
\emph{Claim.}
if $J$ is a Jacobi field along
some $\gamma_s\in\calF$,
and $y:[0,+\infty)\to \R^+$ is a solution of the
differential equation
\begin{equation} \label{eq-jacobi-scalar}
\ddot y = k\,(t+1)^{-3}\,y,
\end{equation}
such that $\norm{J(0)}\leq y(0)$ and
$\snorm{\dot J(0)}\leq \dot y(0)$,
then $\norm{J(t)}\leq y(t)$ for all $t\geq 0$.
\vspace{.2cm}

%\noindent
Let us first prove the claim assuming that the second inequality is strict,
that is to say, that we have $\snorm{\dot{J}(0)}< \dot y(0)$.
In that case we have that the set
\[
\calT=
\set{t>0 \mid \snorm{\dot J (u)}< \dot y(u) \textrm{ for all } u\in [0,t]}
\]
is an open non-trivial interval. Moreover, since $\norm{J(0)}\leq y(0)$,
for any $t\in \calT$ it holds $\norm{J(t)}< y(t)$. 
Hence, thanks to (\ref{borne-HU-hom}) for $t\in\calT$ we have: 
\[
\snorm{\ddot J (t)}=\norm{D^2U(\gamma(t))}\,\norm{\,J(t)}\,<\, k\,(t+1)^{-3}\,y(t)=\ddot y(t).
\]

The last inequality shows that the function
$t\mapsto\dot y(t) - \snorm{\dot J(t)}$ is strictly increasing in $\calT$.
Indeed, given $t_0,t_1\in\calT$, with $t_0<t_1$, by the triangular inequality we have 
\begin{eqnarray*}
    \dot{y}(t_1)-\norm{\dot{J}(t_1)}-\dot{y}(t_0)+\norm{\dot{J}(t_0)}&\geq &\dot{y}(t_1)-\dot{y}(t_0)-\norm{\dot{J}(t_1)-\dot{J}(t_0)} \\
    &\geq &\displaystyle\int_{t_0}^{t_1} (\ddot{y}(s)-\norm{\ddot{J}(s)})\, ds >0.
\end{eqnarray*}
This monotonicity
in turn implies that $\calT=\R^+$, and the claim is true in this case.

Assume now that we have $\norm{J(0)}\leq y(0)$ and $\snorm{\dot J(0)}=\dot y(0)$. 
We can certainly define a sequence of solutions $(y_n)_{n\in\N}$
of the linear differential equation (\ref{eq-jacobi-scalar}),
by imposing the initial conditions
$y_n(0)=y(0)$ and $\dot y_n(0)=\dot y(0)+ 1/(1+n)$ for all $n\in\N$.
Since equation (\ref{eq-jacobi-scalar}) is linear, each $y_n(t)$ is well defined for all $t\geq 0$ and $y_n$ is a convex function on every 
interval where it is positive.   
In addition $\dot{y}_n(0)>0$, thus $y_n(t)$ is always positive for $t\geq 0$.

By the previous argument we know that
$\norm{J(t)}< y_n(t)$ for all $t\geq 0$ and any $n\in\N$.
Since the sequence $y_n(t)$ uniformly converges to $y(t)$ on every compact
subset of $\R^+$, we can say that $\norm{J(t)}\leq y(t)$ for any $t\geq 0$.
Therefore the claim is true in any case.
\vspace{.2cm}

From now on, let $y:[0,+\infty)\to\R$
be the solution of (\ref{eq-jacobi-scalar}) determined
by the initial conditions $y(0)=\dot y(0)=1$, 
which is again positive and well defined in $[0,+\infty)$ for the same reason as before.
Let now $J$ be any Jacobi field along any motion
$\gamma_s\in\calF$, and suppose that  $J(0)$ and $\dot J(0)$ are both
in the unit ball of $E^N$.
Then it follows, from what we have just proved,
that $\norm{J(t)}\leq y(t)$ for all $t\geq 0$. 
Let us set now $x(t)=y(t-1)$, for all $t\geq 1$,
thus $x(t)$ is a solution of $\ddot x=k\,t^{-3}\,x$.
By Lemma \ref{croissance_eq_lin} we know that there is
a constant $c>0$ such that $x(t)\leq ct$ for all $t\ge 1$,
thus $\norm{J(t)}\leq y(t)=x(t+1)\leq c(t+1)$ for all $t\geq 0$.
\end{proof}

\begin{remark}
\label{rem:Jacobi-linear-homog}
Let us endow $E^N\times E^N$ with the norm defined by $\norm{(X,Y)}=\max(\norm{X},\norm{Y})$. 
By linearity of the Jacobi equation,  given a nonzero Jacobi  field $J : [0,+\infty)\to E^N$ along the motion $\gamma_s\in\calF$ with initial condition $(J(0),\dot{J}(0))=Z$, 
we know that $\tilde{J}(t)=\norm{Z}^{-1}J(t)$ is still a Jacobi field, and moreover, $\tilde{J}(0)$ and $\dot{\tilde{J}}(0)$  are both in the unit ball of $E^N$, therefore by 
the previous Lemma we find that $\norm{J(t)}\leq c\norm{Z}(t+1)$, where the constant $c>0$ is independent on the Jacobi field $J$ and on the motion $\gamma_s$.
\end{remark}

The preceding lemma provides a uniform bound for the growth
of Jacobi fields over motions of the $N$-body problem satisfying certain hypotheses.
It is complemented by the following one, which ensures that
these hypotheses are satisfied by hyperbolic motions
whose initial conditions are in a compact set.

\begin{lemma} \label{compacite-fig-lim}
Let $\calF$ be a family of hyperbolic motions defined for $t\geq 0$.
If the set of initial conditions
\[
\calB=\{(\gamma(0),\dot{\gamma}(0))\;\mid\; \gamma\in \calF \}
\]
is compact in $\Omega$, then the hypothesis of Lemma \ref{croiss-jacobi} holds,
that is, there is a compact set $K\subset\Omega$, and a constant $b>0$
such that for any $\gamma\in \calF$ and any $t\geq 0$ we have
\begin{enumerate}
\item[(i)]
$\gamma(t)/\norm{\gamma(t)}\in K$, and
\item[(ii)]
$\norm{\gamma(t)}\geq b(t+1)$.
\end{enumerate}   
\end{lemma}

 \begin{remark}
In assumption (i) we are essentially asking that all the normalized configurations stay sufficiently far away from collisions.
\end{remark}

\begin{proof} 
Let us prove (i). The set 
\[
K_0=\set{\gamma(t)/\norm{\gamma(t)}\;\mid\;
\gamma\in\calF,\; t\geq 0}
\]
is trivially bounded, thus the closure $K$ of $K_0$
is a compact subset of $E^N$. To check that (i) holds
we only have to show that $K\subset\Omega$.

In order to show this, let $y\in K$,
a sequence of motions $\gamma_n\in \calF$
and a sequence of times $t_n\in[0,+\infty)$ such that 
\[
y=\lim\limits_{n\to +\infty}
\gamma_n(t_n)/\norm{\gamma_n(t_n)}.
\]
Let us denote $z_n=(\gamma_n(0),\dot{\gamma}_n(0))$,
thus $z_n$ is a sequence in $\calB$  and for each $t\geq 0$
it holds $\gamma_n(t)=\pi\circ\varphi^t(z_n)$.   
Without loss of generality, we can assume that $z_n$ converges
to some configuration $z\in\calB$.
Let us set $\gamma(t)=\pi\circ\varphi^t(z)$.

We now discuss whether the sequence $(t_n)$ is bounded or not.
If $(t_n)$ is bounded, up to extracting
we can assume that $t_n$ converges to some $t^*\geq 0$.
In this case we have $y=\gamma(t^*)/\norm{\gamma(t^*)}$,
hence $y\in\Omega$ because $\calB\subset\calH$.
If $(t_n)$ is unbounded, up to extracting
we can assume that $t_n\to +\infty$ as $n\to +\infty$.  
In this case, let us denote
\[
a=\sm(z)=\lim\limits_{t\to +\infty}
\frac{\gamma(t)}{t}
\]
By item (2) of Lemma \ref{lem:lema-cont.limitshape} (Lemma 4.1 in \cite{MaVe}), for each $\epsilon>0$, there is $\tau>0$ such that for all $n$ big enough and for all $t\geq \tau$ it holds
\[
\norm{\gamma_n(t)-ta}\leq \epsilon t,
\]
which implies
\[
a=\lim\limits_{n\to +\infty}\sm(z_n)=
\lim\limits_{n\to +\infty}\frac{\gamma_n(t_n)}{t_n}.
\]
Since $\norm{a}=\lim_{n\to +\infty}\norm{\gamma_n(t_n)}/t_n>0$, we conclude that
\[
y=\lim\limits_{n\to +\infty}
\;\gamma_n(t_n)/\norm{\gamma_n(t_n)}=\,a/\norm{a}\in\Omega.
\]
Let us prove now that (ii) holds. If $z\in \calB$,
applying again Lemma \ref{lem:lema-cont.limitshape}  with
$\epsilon=\norm{\sm(z)}/2$,
we can find $\delta=\delta(z)>0$ and
$\tau=\tau(z)>0$ such that for any
$w\in B(z,\delta)$ and $t\ge \tau$ it holds
\[
\norm{\pi\circ\varphi^t(w)}\geq \;t\,\norm{\mathfrak a(z)}/2.
\]
From now on, if $w\in\calB$,
we write $x_w(t)$ to denote the motion
given by $\pi\circ\varphi^t(w)$.
By compactness of $\calB$, we find a finite open cover ${U_1,\dots, U_n}$ of $\calB$,
and constants $T_i>0$ and $R_i>0$, for each $i=1,\dots, n$, such that
\[
\norm{x_w(t)}\geq t R_i
\]
whenever $w\in\calB\cap U_i$ and $t>T_i$.
In particular, if we set $R=\min\set{R_1,\dots,R_n}$
and $T=\max\set{T_1\dots,T_n}$ we have
\[
\norm{x_w(t)}\geq tR
\]
whenever $w\in\calB$ and $t>T$. 
On the other hand, the set 
\[
\{ x_w(t)\quad  \mid\quad w\in \calB,\ t\in [0,T] \} 
\]
is compact and contained in $\Omega$, thus
there is a constant $S>0$ such that the inequality
\[
\norm{x_w(t)}\geq (T+1)S\geq(t+1)S
\]
holds for $w\in\calB$ and $0\leq t\leq T$.
Then there is $b>0$ such that
\[
\norm{x_w(t)}\geq b(t+1)
\]
for all $w\in\calB$ and all $t\geq 0$.

\end{proof}

\subsection{Regularity of the limit shape map}
\label{ss:regul-lim-shape}

We are now ready to state a regularity result of the limit shape map.

\begin{remark}
As in the section \ref{ss.Jacobi.equ},
for a given configuration without collisions $x\in\Omega$
we will denote $D^2U(x)$ the Hessian of $U$ at $x$, that is the endomorphism of $E^N$
which represents the symmetric bilinear form $d^{\,2}_xU$
with respect to the mass inner product.
\end{remark}

\begin{proposition} \label{differ-fig-limite}
The limit shape map $\sm:\calH\to E^N$ is of class $\calC^1$. 
Moreover, the following representation formula holds:

If $z=(x_0,v_0)\in \calH$ and $Z=(X,V)\in E^N\times E^N$, then
\[
d_z\sm(Z)=V+\int_0^{+\infty}
D^2U(x(t))\,J(t)\, dt,
\]
where $x(t)$ is the hyperbolic motion
with initial conditions $(x_0,v_0)$
and $J(t)$ is the Jacobi field along $x(t)$
with initial condition $(J(0),\dot J(0))=(X,V)$.
\end{proposition}
\begin{proof}

We have already remarked that if $y : [0,+\infty)\to \Omega$ is a hyperbolic motion, the limit shape of $y$ is given by 
$\lim_{t\to +\infty} \dot{y}(t)$, and 
since $y$ is a motion, we have 
\[
\sm(y(0),\dot{y}(0))=\dot{y}(0)+\displaystyle\int_0^{+\infty} \nabla U(y(t))\, dt.
\]
Let us choose $\epsilon>0$ such that the segment
$[z-\epsilon Z,z+\epsilon Z]$ is contained in $\calH$.
We know that
for any $s\in [-\epsilon,\epsilon]$, if we set  
\[
x(s,t)=
\pi\circ\varphi^t(x_0+sX, v_0+sV)\,.
\]
then it holds
\[
{\mathfrak a}(z+sZ)=
v_0+sV+\displaystyle\int_0^{+\infty} \nabla U(x(s,t))\, dt.
\]
In a first time, we think at $z$ and $Z$ as being fixed,
and we wish to show the existence of the directional derivative
\[
\partial_Z\,{\mathfrak a}(z)=\left.\frac{d}{ds}\right|_{s=0} {\mathfrak a}(z+sZ)
\]
so we need to show that
\[
\frac{d}{ds}
\int_0^{+\infty} \nabla U(x(s,t))\,dt=
\int_0^{+\infty} \frac{\partial}{\partial s}\nabla U(x(s,t))\,dt
\]
and for this, we will find a uniform domination condition for
\[
\frac{\partial }{\partial s}\nabla U (x(s,t))\,=\;D^2U(x(s,t))\;J_s(t)
\]
where $J_s$ is the Jacobi field along $t\mapsto x(s,t)$ such that
$J_s(0)=X$ and $\dot J_s(0)=V$.

By applying Lemma \ref{compacite-fig-lim}
to the compact set of initial conditions given by $\calB=[z-\epsilon Z,z+\epsilon Z]$, we find a compact subset $K\subset\Omega$ and a constant $b>0$ such that for all
$s\in [-\epsilon,\epsilon]$ and $t\geq 0$ we have
\[
x(s,t)/\norm{x(s,t)}\in K,\quad \norm{x(s,t)}\geq b(t+1).
\]
Since $D^2U$ is homogeneous of degree $-3$ and bounded on $K$, there exists a strictly positive constant $k>0$ such that for $s\in [-\epsilon,\epsilon]$
and $t\geq 0$ it holds
\[
\norm{D^2U(x(s,t))}\leq\, k\,(t+1)^{-3}
\]
Moreover, as noticed in Remark \ref{rem:Jacobi-linear-homog}, we find $c>0$
such that for any $s\in [-\epsilon,\epsilon]$ and $t\geq 0$
we have
\[
\norm{J_s(t)}\leq \, c\,\norm{Z}\,(t+1),
\]
from which we get the domination
\[
\norm{\frac{\partial}{\partial s}
\nabla U(x(s,t))}\;
\leq \;kc\norm{Z}(t+1)^{-2},
\]
for $s\in [-\epsilon,\epsilon]$ and $t\ge 0$.
Since the function $t\mapsto (t+1)^{-2}$ is integrable over
$[0,+\infty)$, we deduce by the dominated convergence theorem that $s\mapsto \sm(z+sZ)$ 
is differentiable. The derivative at $s=0$ is precisely
\[
\partial_Z  \sm(z)=
V+\int_0^{+\infty} D^2U(x(t))\,J(t)\,dt\,.
\]

Now, in order to prove that the limit shape map is $\calC^1$
we have to prove that for any fixed $Z=(X,V)\in E^N\times E^N$,
the partial derivative map
\[
z\;\mapsto\; \partial_Z {\mathfrak a}(z)
\]
is continuous on $\calH$. Reasoning as before,
applying Lemmas \ref{croiss-jacobi} and \ref{compacite-fig-lim} to
a given compact subset $K\subset \calH$,
we find again a positive constant $d=d(K)>0$ such that,
if $z\in K$, $Z\in E^N\times E^N$, and $t\geq 0$, then we have
\[
\norm{D^2U(x(t)) J(t)}\leq \,d \,\norm{Z}\,(t+1)^{-2},
\]
where $x(t)=\pi\circ\varphi^t(z)$ and $J$ is the Jacobi field
along $x$ with initial condition $(J,\dot J)(0)=Z$.
Applying again the dominated convergence theorem, we obtain the continuity of the map 
\[
K\to \R,\quad
z\;\mapsto \;\int_0^{+\infty} D^2U(x(t)) J(t)\,dt.
\]
We note that in this expression,
both $x(t)$ and $J(t)$ depends continuously on $z$.
Since $K$ is an arbitrary compact subset of $\calH$,
we deduce that $z\mapsto \partial_Z {\mathfrak a}(z)$
is continuous on $\calH$, thus ${\mathfrak a}$ is $\calC^1$,
and $d_z{\mathfrak a}(Z)=\partial_Z {\mathfrak a}(z)$
as we wanted to prove.

\end{proof}

\section{Uniqueness of hyperbolic motions}
\label{sec:uniq.hyp.motions}

In this section we prove that if $a\in\Omega$ is a configuration without collisions, there exists
a cutted cone $\cone=\cone_a(\alpha,r)$
with the following property: for each point $x_0\in\cone$ there is a unique hyperbolic 
motion contained in $\cone$, starting at $x_0$ at time $t=0$
whose limit shape is the configuration $a$.

\begin{lemma} \label{exist-applic-figlim}
Let $a\in\Omega$ be a configuration without collisions
and $\alpha>0$ such that the cone $\cone_a(\alpha)$
is contained in $\Omega\cup\set{0}$.
Two constants $\delta>0$ and $\lambda>0$ can be fixed
in order that the following property holds:
for any $\epsilon>0$ there is $r_0>0$ big enough such that,
if $z=(x_0,v_0)\in\cone_a(\alpha,r_0)\times\overline{B}(a,\delta)$
then we have
\begin{itemize}
\item[(i)]
The motion $x(t)=\pi\circ\varphi^t(z)$ is defined for all $t\geq 0$,
moreover we have
\[
x(t)\in\cone_a(\alpha,r_0)\;\textrm{ and }\;
\norm{\dot{x}(t)-v_0}\leq\epsilon\,.
\]
\item[(ii)] $t\mapsto \norm{x(t)}$ is strictly increasing and
$\norm{x(t)}\geq \norm{x_0}+\lambda t$.
\item[(iii)] $z\in\calH$ and
\[
\norm{\sm(z)-v_0}\leq \epsilon\,.
\]
\end{itemize}      
\end{lemma}
    
\begin{proof}
First, we want to show that there is $\delta>0$ such that,
if $\norm{\dot x(t)-v_0}<\delta$ for all $t\in(0,\tau)$ and
for some $\tau>0$ then in fact both $x(t)$ and $\dot x(t)$
remain in the cone $\cone_a(\alpha)$.
This is fulfilled if we take $\delta>0$ small enough such that
the closed ball $\overline B(a,2\delta)$ is contained in
the cone $\cone_a(\alpha)$.
Indeed, since we have  $\norm{v_0-a}\leq\delta$  then
we also have $\norm{\dot x(t)-a}<2\delta$ for all $0\leq t<\tau$.
Since the cone is additively closed, it follows that
\[
x(t)=x_0+\int_0^t\dot x(s)\,ds
\]
is also in the cone $\cone_a(\alpha)$ for all $t<\tau$.

Now we want to ensure that $\norm{\dot x(t)-v_0}$ must remain smaller than $\delta$,
so we want to estimate the integral
\[
\dot x(t)-v_0=\int_0^{t}\nabla U(x(s))\,ds\,,
\]
which can be achieved, since $\nabla U$ is homogeneous of degree $-2$,
by ensuring rapid growth in the size of $x(t)$. For this reason,
we have to prove at the same time both statements (i) and (ii).

In order to do this, let $\delta>0$ such that
the closed ball $\overline B(a,2\delta)$ is contained in
the cone $\cone_a(\alpha)$,
let $[0,T)$ be the maximal interval of definition in the future
for the motion $x(t)$, and let
\[
\tau=\sup\set{t\in(0,T)\mid
\norm{\dot x(s)-v_0}<\delta \textrm{ for all } 0<s<t}\,.
\]
Let also denote $I(t)=\norm{x(t)}^2$. For $t\in [0,\tau)$ we have
$x(t)\in\cone_a(\alpha)$,
which can be expressed by $\inner{x(t)}{a}\geq \alpha\norm{a}\norm{x(t)}$,
and therefore
\begin{eqnarray*}
\dot I (t)&=&2\inner{x(t)}{\dot{x}(t)}=
2(\inner{x(t)}{a}+\inner{x(t)}{\dot{x}(t)-a})\\
&\geq& 2(\alpha\norm{a}-2\delta)\;\norm{x(t)}.
\end{eqnarray*}
This last inequality can be written as
\[
\dot I(t)\,I(t)^{-1/2}\geq 2\lambda\,,
\]
where $\lambda=\alpha\norm{a}-2\delta$. 
From now on we suppose also that $\delta<\alpha\norm{a}/2$
so we have that $\lambda>0$,
which gives the strict growth of $\norm{x(t)}$ in $[0,\tau)$.
By integrating we find 
\[
\norm{x(t)}=I(t)^{1/2}\geq \norm{x_0}+\lambda t 
\]
for all $t\in[0,\tau)$.

To complete the proof of (i), we recall that we have a given
$\epsilon>0$, and we must choose $r_0>0$
big enough in such a way that $\tau=T=+\infty$, and that
\[
\norm{\dot x(t)-v_0}\leq
\int_0^{t}\norm{\nabla U(x(s))}\,ds\,\leq\epsilon
\]
for all $t\geq 0$,
provided $x_0\in\cone_a(\alpha,r_0)$ and $\norm{v_0-a}\leq \delta$.
If we define
\[
\mu=\sup
\set{\norm{\nabla U(x)}\mid
x\in\cone_a(\alpha) \textrm{ and }\norm{x}=1}
\]
we can write, for $t\in [0,\tau)$,
\[
\norm{\dot x(t)-v_0}\,\leq\,
\int_0^{+\infty}\frac{\mu}{(\norm{x_0}+\lambda s)^2}\,ds\,=\,
\frac{\mu}{\lambda \norm{x_0}}\,.
\]
Thus, if we define
\[
r_0=\frac{2\mu}{\,\lambda\,\min\set{\epsilon,\delta}\,}\,,
\]
we get
\begin{equation}\label{eq.lemma3.1}
\norm{\dot x(t)-v_0}\leq\frac{1}{2}\min\set{\epsilon,\delta}
\end{equation}
for all $t\in [0,\tau)$.
Therefore we have proved that, for this choice of $r_0>0$,
if $\norm{\dot x(t)-v_0}<\delta$ for all $t\in [0,\tau)$
then actually $\norm{\dot x(t)-v_0}\leq \delta/2$ for $t\in[0,\tau)$.
Since $\dot x(t)$ is continuous, the set
$[0,\tau)$ is relatively closed in $[0,T)$ so $\tau=T$.

Now we see that $T=+\infty$. We have, 
$\norm{\dot x(t)-v_0}<\delta$ for $t\in [0,T)$,
which in turn implies that $x(t)\in\cone_a(\alpha,r_0)$
also for any $t\in [0,T)$.
Since $\cone_a(\alpha,r_0)\subset \Omega$ and $\nabla U$ is a homogeneous function of degree $-2$, it is bounded on the cone $\cone_a(\alpha,r_0)$, therefore $x(t)$ is defined for all $t\geq 0$, that is $T=+\infty$.

We have proved (i) and (ii). Then we can deduce that
the integral of $\nabla U(x(t))$
converges,  meaning also that $\dot x (t)$ converges
to a configuration without collisions, since by assumptions on $v_0$ and estimate (\ref{eq.lemma3.1}) we have that $\dot{x}(t)\in \overline{B}(a,2\delta)$ for all $t\geq 0$, 
and by choice of $\delta$ we have $\overline{B}(a,2\delta)\subset \Omega$.  
Therefore $x(t)$ is a hyperbolic motion,
that is $z\in\calH$, and moreover $\sm(z)$ satisfies (iii).
\end{proof}

The following Proposition gives a proof of item $(1)$
in Theorem \ref{thm:cone}.
\begin{proposition} \label{thm-unicite-loc}
Let $a\in\Omega$ be a configuration without collisions and $\cone_a(\alpha)$ be a given cone contained in $\Omega\cup\{0\}$. Then there exists 
$r_1>0$ such that,
for any $x_0\in\cone_a(\alpha,r_1)$
there exists a unique hyperbolic motion $x(t)$ starting at $x_0$,
with limit shape $a$, and contained in the cone 
$\cone_a(\alpha,r_1)$ for all $t\geq 0$.
Moreover, given $\epsilon>0$ we can choose $r_1>0$
in order that $\norm{\dot{x}(t)-a}<\epsilon$
for all $t\geq 0$.
\end{proposition}
\begin{proof}
We start by applying Lemma \ref{exist-applic-figlim} with $\epsilon=\delta/4$.
Thus we are assuming that there are constants $\delta>0$ and
$r_0>0$ for which the following property holds:
if $x_0\in\cone_a(\alpha,r_0)$ and $v_0\in\overline B(a,\delta)$ then
$(x_0,v_0)\in\calH$, and $\norm{\sm (x_0,v_0)-v_0}\leq \delta/4$.

The next step is to prove the existence of 
an implicit function for the equation $\sm (x_0,v_0)=a$ in some cone
$\cone_a(\alpha, r_1)$ with $r_1>r_0$.
That is, we are going to prove the following claim.

\vspace{.2cm}
\noindent
\emph{Claim.}
For some $r_1\geq r_0$ there is
a $C^1$ function $\calV_a:\cone_a(\alpha,r_1)\to B(a,\delta)$
such that for each $x_0\in\cone_a(\alpha,r_1)$ we have that $v_0=\calV_a(x_0)$ is the unique initial velocity
in $\overline B(a,\delta)$ satisfying $\sm(x_0,v_0)=a$.
\vspace{.2cm}

The proof of this claim will use the fact that,
for some $r_1>0$ big enough, we can make
the partial derivative $\partial \sm/\partial v$
as close as we want to the identity in $E^N$, and this uniformly in
$z=(x_0,v_0)\in\cone_a(\alpha,r_1)\times\overline B(a,\delta)$.

Indeed, according to Proposition \ref{differ-fig-limite}
we know that for a given $V\in E^N$
\[
\frac{\partial\,\sm }{\partial\,v}(z)(V)=
d_z\sm(0,V)=
V+\int_0^{+\infty} D^2U(x(t))\,J(t)\,dt
\]
where $x(t)$ is the hyperbolic motion
with initial condition $z=(x_0,v_0)$
and $J$  is the Jacobi field along the motion $x(t)$
with initial conditions $J(0)=0$ and $\dot J(0)=V$.
In order to get a uniform estimate for the integral,
we will apply Lemma \ref{croiss-jacobi}  
with the compact set
$K=\set{x_0\in\cone_a(\alpha)\mid \norm{x_0}=1}$ and Remark \ref{rem:Jacobi-linear-homog}.
Note that condition (ii) in this lemma
is a consequence of item (ii) in
Lemma \ref{exist-applic-figlim} by choosing
$b=\min\set{\lambda,r_0}$. Therefore
we arrive to the following upper bounds:
\begin{eqnarray*}
\norm{J(t)}&\leq & c\,\norm{V}\,(t+1)\,\\
\norm{D^2U(x(t))}&\leq& d\,(\norm{x_0}+\lambda t)^{-3}
\end{eqnarray*}
for some constants $c,d>0$ depending only on $r_0$
and the cone $\cone_a(\alpha)$.
Then, given $\eta>0$ we have
\[
\norm{\int_0^{+\infty} D^2U(x(t))\,J(t)\,dt}\;\;\leq\;\;
\frac{cd\norm{V}}{\norm{x_0}}\;\int_0^{+\infty}
\frac{u+r_0^{-1}}{(1+\lambda u)^3}du
\leq  \eta\norm{V}
\]
provided that $\norm{x_0}\geq r_1=\max\{r_0,\rho_0\}$,
where we set 
\[
\rho_0=
\frac{cd}{\eta}\int_0^{+\infty}
\frac{u+r_0^{-1}}{(1+\lambda u)^3}du\,.
\]
We conclude that for $z=(x_0,v_0)$ in the product
$\cone_a(\alpha,r_1)\times\overline B(a,\delta)$
we have that for any $V\in E^N$
\[\norm{\frac{\partial\,\sm}{\partial\,v}(z)(V)-V}
\leq \eta\norm{V},\]
meaning that
\[
\norm{\frac{\partial\,\sm}{\partial\,v}(z)-Id}<\eta
\]
for any
$z\in\cone_a(\alpha,r_1)\times\overline B(a,\delta)$.
This implies the injectivity of the map
\[
v\;\mapsto\; \sm(x_0,v)
\]
in $\overline B(a,\delta)$ for all $x_0\in\cone_a(\alpha,r_1)$,
provided that $\eta<1$.
Indeed, if $v_0,w_0\in\overline B(a,\delta)$
and $x_0\in\cone_a(\alpha,r_1)$ then
\[
\norm{\sm(x_0,w_0)-\sm(x_0,v_0)}\geq (1-\eta)\norm{w_0-v_0}\,.
\]

Let us prove now that for any $x_0\in \cone_a(\alpha,r_1)$,
there exists $v_0\in B(a,\delta)$ such that 
$\sm(x_0,v_0)=a$. If this sentence is true, 
then $\calV_a(x_0)$ is defined as the unique $v_0$
contained in $B(a,\delta)$ satisfying the equation
$\sm(x_0,v_0)=a$. 
The fact that the map $\calV_a$ is $\calC^1$
is therefore a straightforward consequence of
the implicit function theorem.

Given $x_0\in\cone_a(\alpha,r_1)$, let us define the function
\[
g_{x_0}:\overline{B}(a,\delta)\to \R,
\qquad g_{x_0}(v)=\norm{\sm(x_0,v)-a}^2.
\]
Recall that by choice of $r_1$ we have that
$\norm{\sm(x_0,v)-v}\leq \delta/4$
for all $v\in\overline{B}(a,\delta)$, 
and in particular  $\norm{\sm(x_0,a)-a}\leq \delta/4$. 
At the same time,
if $v\in \partial\overline{B}(a,\delta)$ then
\[
\norm{\sm(x_0,v)-a}\geq
\norm{v-a}-\norm{\sm(x_0,v)-v}\geq \delta-\delta/4
=\frac{3\delta}{4},
\]
in particular, the minimum of $g_{x_0}$ is never reached on
$\partial\overline{B}(a,\delta)$.
Let now $v_0\in B(a,\delta)$ be a point where $g_{x_0}$
reaches a minimum. Assume for the sake of contradiction
that $\sm(x_0,v_0)\neq a$. We know that for any $V\in E^N$
it holds
\[
d_{v_0}g_{x_0}V\;=\;
2\;\inner{\;\sm(x_0,v_0)-a\;}{\;\frac{\partial\,\sm}{\partial\,v}(x_0,v_0)V\;}=0,
\]
but since $\partial\sm/\partial v(x_0,v_0)$
is invertible, we can find $V\in E^N$ so that
\[
\frac{\partial\,\sm}{\partial\,v}(x_0,v_0)\,V=
{\sm}(x_0,v_0)-a.
\]
Thanks to this choice we find
\[
d_{v_0}\,g_{x_0}\,V=2\norm{\sm(x_0,v_0)-a}^2=0,
\]
which contradicts the assumption $\sm(x_0,v_0)\neq a$.
Thus the equation $\sm(x_0,v_0)=a$ is satisfied at some
$v_0\in B(a,\delta)$. This finishes the proof of the claim.

\vspace{.2cm}

It follows now that for $x_0\in\cone_a(\alpha,r_1)$
the function $t\mapsto x(t)=\pi\circ\varphi^t(x_0,\calV_a(x_0))$
is the unique hyperbolic motion with limit shape $a$,
among all solutions starting at $x_0$ at time $t=0$
and whose initial velocity is in the ball $\overline B(a,\delta)$.
By (i) and (ii) of Lemma \ref{exist-applic-figlim},
the motion  $x(t)$ is contained in the cone $\cone_a(\alpha,r_1)$.
Let us assume now, for the sake of contradiction,
the existence of another hyperbolic motion $y(t)$,
starting at some $x_0\in \cone_a(\alpha,r_1)$ at $t=0$,
with limit shape $a$, also contained in the cutted cone
$\cone_a(\alpha,r_1)$, and satisfying
$\dot y(0)\notin\overline B(a,\delta)$.
Necessarily we have $\lim_{\,t\to +\infty}\dot y(t)=a$,
hence there must be a time $\tau>0$ for which
$\dot y(\tau)\in\partial\overline B(a,\delta)$.
Clearly
$\sm(y(\tau),\dot y(\tau))=a$, hence by the the previous claim
$\dot y(\tau)=\calV_a(y(\tau))$, but this is in contradiction with
the fact that $\calV_a$ takes values
in the open ball $B(a,\delta)$. 

Finally, given $\epsilon>0$, we apply the previous claim and Lemma \ref{exist-applic-figlim} 
with $\epsilon$ replaced with $\epsilon/2$. Thus, writing 
again $x(t)=\varphi^t(x_0,\calV_a(x_0))$ for $x_0\in \cone_a(\alpha,r_1)$, 
since $x(0)=x_0$ and $\sm(x(0),\dot{x}(0))=a$,
we get that $\norm{\dot{x}(0)-a}<\epsilon/2$ as well as that
$\norm{\dot{x}(t)-\dot{x}(0)}\leq \epsilon/2$ for all $t\geq 0$,
therefore $\norm{\dot{x}(t)-a}<\epsilon$ for all $t\geq 0$
as we wanted to prove.
\end{proof}

\begin{remark} \label{rk-monot}
For the motions $x(t)$ considered in Lemma
\ref{exist-applic-figlim} and Proposition \ref{thm-unicite-loc},  
the size $\norm{x(t)}$ is always a strictly increasing function.
Therefore, the conclusions of the two results still holds
if we replace $r_0$ or $r_1$ with bigger constants.
In a similar way, in Lemma \ref{exist-applic-figlim},
if we keep fixed $\epsilon$ and we decrease $\delta$,
the conclusions are still true with the same constants $\lambda$ and $r_0$.  
\end{remark}

\section{Proof of the theorem of the cone}
\label{sec:proof.thm.cone}

Given a configuration $a\in\Omega$, let us recall that
$\alpha_0(a)$ is the infimum of the $\alpha\in (0,1)$ such that
$\cone_a(\alpha)\cap\Delta=\set{0}$.
Given $\alpha\in(\alpha_0(a),1)$,
by Proposition \ref{thm-unicite-loc} we can find
a constant $r_1>0$ such that from each $x_0\in \cone_a(\alpha,r_1)$ 
it starts a unique hyperbolic motion contained in
the cutted cone $\cone_a(\alpha,r_1)$ with limit shape $a$.
Now we show that, in an eventually reduced cone,
the uniqueness of the hyperbolic motion
(starting in the cone, and with limit shape $a$) is maintained
if we replace the condition of being contained in the cone
with the condition of being minimizing.
In other words, what distinguishes the unique hyperbolic
motion contained in the cone from the others
(that eventually may exist) is the property of being a geodesic ray.
The strategy of the proof is as follows.
We choose $\alpha$ and $\beta$ satisfying
 $\alpha_0(a)<\beta<\alpha<1$,
thus $\cone_a(\alpha)\subset\cone_a(\beta)$,
and $r_1>0$ such that the conclusion of
Proposition \ref{thm-unicite-loc} holds for both
$\cone_a(\alpha,r_1)$ and $\cone_a(\beta,r_1)$.
Then we will find a constant $r_2\geq r_1$
such that if $x\in\cone_a(\alpha,r_2)$,
then any hyperbolic geodesic ray $\gamma\in\calM_a$
starting at $\gamma(0)=x$ is necessarily contained
in $\cone_a(\beta,r_1)$.
This will give the uniqueness of $\gamma$.
We state before some preliminary results.

\begin{lemma}\label{lemma:D-estim}
Let $h>0$ and $a\in\Omega_h$. For any given $\beta\in(\alpha_0(a),1)$, three positive constants $k,l,m>0$ can be chosen
so that the following holds : if $x,y,z\in\cone_a(\beta)$, $\rho>0$ and $\lambda\in (0,1)$ satisfy
\[
\norm{x}=\norm{z}=\rho\,,\;\textrm{ and }
\;\norm{y}=\lambda\rho\,,
\]
then we have
\[
\calD(x,y,z)=\phi_h(x,y)+\phi_h(y,z)-\phi_h(x,z)
>k(l-\lambda)\rho-m\,.
\]
\end{lemma}

\begin{proof}
By a standard argument (see Lemma 4.4 in \cite{MaVe})
we know that
\begin{equation*}
\phi_h(x,y)\geq\sqrt{2h}\norm{y-x}
\;\;\textrm{ and }\;\; 
\phi_h(y,z)\geq\sqrt{2h}\norm{z-y}.
\end{equation*}

To establish the lemma, we will now estimate $\phi_h(x,z)$
using the linear path going from $x$ to $z$ with constant speed
equal to $\sqrt{2h}$.
If $\sigma:[0,T]\to\Omega$ is that linear path then
clearly we have
\[
T=\frac{\norm{z-x}}{\sqrt{2h}}\;\;
\textrm{ and }\;\;
\dot\sigma(t)=\sqrt{2h}\;\frac{z-x}{\norm{z-x}}
\]
for all $t\in [0,T]$. Therefore we obtain the upper bound
\[
\phi_h(x,z)\;\leq\;
\calA_{L+h}(\sigma)=\sqrt{2h}\,\norm{z-x}+
\int_0^T U(\sigma(t))\,dt.
\]
We now look for an upper bound for the integral of
the Newtonian potential over the segment $\sigma$.
According to the definition of the cone $\cone_a(\beta)$,
we also know that
\[
\norm{z-x}\leq \,2\rho\,\sqrt{1-\beta^2},
\;\;\textrm{ and that }\;\;
\norm{\sigma(t)}\geq \,\rho\,\beta
\]
for all $t\in [0,T]$. Thus, if we set
\[
\mu=\max\set{U(w)\mid\;w\in \cone_a(\beta), \;\norm{w}=1},
\]
by using the homogeneity of the potential we get
the upper bound
\[
\int_0^T U(\sigma(t))\,dt\leq \int_0^T \frac{\mu}{\norm{\sigma(t)}}\, dt \leq
T\,\frac{\mu}{\rho\beta}=
\frac{\norm{z-x}}{\sqrt{2h}}
\frac{\mu}{\rho\beta}\leq m,
\]
where we set
\[
m=\frac{2\mu}{\sqrt{2h}}\sqrt{\frac{1}{\beta^2}-1}.
\]
On the other hand, we also have that
\[
\min(\norm{x-y},\,\norm{z-y})\,\geq\,(1-\lambda)\rho\,,
\]
from which we deduce that
\[
\calD(x,y,z)\;\geq\;
2\sqrt{2h}
\left(\left(1-\sqrt{1-\beta^2}\right)-\lambda\right)
\rho\,-\,m\,.
\]
So the proof is achieved by setting also
$k=2\sqrt{2h}$ and $l=1-\sqrt{1-\beta^2}$.
\end{proof}

\begin{proposition}
\label{prop:borne-rentr}
Given $\beta\in(\alpha_0(a),1)$ there exists $\lambda>0$
and $r>0$ such that, if $\gamma\in\calM_a$
is a hyperbolic geodesic ray starting at $x\in\cone_a(\beta,r)$,
and $y$ is the configuration where $\gamma$ enters definitively
into $\cone_a(\beta,r)$, then $\norm{y}\geq\lambda\norm{x}$. 
\end{proposition}

\begin{proof}
If $\gamma\in\calM_a$ with $\gamma(0)=x$,
then the configuration $y$ where $\gamma$ enters definitively
into the cone is $y=\gamma(\tau)$ where
\[
\tau=\min\set{t\geq 0\mid
\gamma(s)\in\cone_a(\beta,r) \textrm{ for all }s>t}.
\]

We now discuss according to $\norm{y}\geq\norm{x}$ or not.
In the first case, any $\lambda\leq 1$ works.
Therefore we assume now that $\norm{y}<\norm{x}$, which
is the non-trivial case.

We observe that for some $\tau'>\tau$ we must have
$\norm{\gamma(\tau')}=\norm{x}$. So the restriction of $\gamma$
to $[0,\tau']$ is a free-time minimizer for the action $\calA_{L+h}$  with extremities
$x=\gamma(0)$ and $z=\gamma(\tau')$, in particular it is a minimizing geodesic for the distance $\phi_h$.
Moreover, the configuration $y=\gamma(\tau)$
belongs to this minimizing segment,
hence
\[
\phi_h(x,y)+\phi_h(y,z)=\phi_h(x,z)
\]
or equivalently
\[
\calD(x,y,z)=\phi_h(x,y)+\phi_h(y,z)-\phi_h(x,z)=0\,.
\]
Taking into account that $\norm{x}=\norm{z}\geq r$,
this equality contradicts Lemma \ref{lemma:D-estim}
if we take, for instance,
\[
\rho=\norm{x}=\norm{z},\;\;
\lambda=\frac{l}{2}\;\;\textrm{ and }\;\;r>\frac{2m}{kl}
\]
where $k,l,m>0$ are the positive constants
provided by Lemma \ref{lemma:D-estim}
\end{proof}

\begin{proposition} \label{prop:bornes-config-rentr}
Let $a\in\Omega$, and let $\alpha$ and $\beta$ such that
$\alpha_0(a)<\beta<\alpha<1$.
There exist $r_1>0$ large enough, and $c\geq 1$ such that,
for any $x\in\cone_a(\beta,r_1)\setminus\cone_a(\alpha)$,
the unique hyperbolic motion $\gamma$ starting at $x$,
with limit shape $a$, and contained in $\cone_a(\beta,r_1)$
enters in $\cone_a(\alpha)$ at a configuration $w$ which satisfies
$\norm{w}\leq c\norm{x}$.
\end{proposition}

\begin{proof}
Let $\alpha'\in (\alpha,1)$ and $\epsilon>0$ such that
$B(a,\epsilon)\subset\cone_a(\alpha')$. 
For this $\epsilon$ we apply Proposition
\ref{thm-unicite-loc} to the cone $\cone_a(\beta)$.
Then we get $r_1>0$ such that, if $x\in\cone_a(\beta,r_1)$
then there is a unique hyperbolic motion $\gamma$
with limit shape $a$, starting at $x$ and contained
in $\cone_a(\beta,r_1)$. Moreover, it satisfies
$\dot \gamma(t)\in B(a,\epsilon)\subset\cone_a(\alpha')$
for all $t\geq 0$.
Therefore if $0\leq s\leq t$ we still have the inclusion
$\gamma(t)-\gamma(s)\in\cone_a(\alpha')$, in particular
$\gamma(t)-x\in \cone_a(\alpha')$ for all $t\geq 0$.
We will show now that this choice of $r_1$ gives 
the desired constant $c\geq 1$.  

Assuming that
$x=\gamma(0)\in\cone_a(\beta,r_1)\setminus\cone_a(\alpha)$,
let $w=\gamma(\tau)$ be the the first point where $\gamma$
enters the cone $\cone_a(\alpha)$.
That is,
\[
\tau=\min\set{t>0\mid\gamma(t)\in\cone_a(\alpha)}>0.
\]
We observe that, since $\alpha'>\alpha$,
for any $w\in\cone_a(\alpha)$ we have
\[
w+\cone_a(\alpha')\subset\cone_a(\alpha).
\]
Since $\dot{\gamma}(t)\in\cone_a(\alpha^\prime)$ for all $t\geq\tau$ and the cone $\cone_a(\alpha^\prime)$ is additively closed, it follows that 
$\gamma(t)\in w+\cone_a(\alpha^\prime)\subset \cone_a(\alpha)$ for $t\geq \tau$, which means that once $\gamma$ enters the cone,
it does so definitively. We recall that we are looking for an upper bound of
$\norm{w}/\norm{x}$ that does not depends on
$x\in\cone_a(\beta,r_1)\setminus\cone_a(\alpha)$.
To obtain this bound, we will forget for a moment
the curve $\gamma$, and only retain the following essential information :
we have
\begin{enumerate}
    \item[(i)]$x\in\cone_a(\beta)\setminus\cone_a(\alpha)$,
    so in particular $\norm{x}>0$,
    \item[(ii)]$w\in\partial\,\cone_a(\alpha)$,
    the boundary of the cone, and
    \item[(iii)]$w-x\in\cone_a(\alpha')$. 
\end{enumerate}
Clearly these three conditions are preserved by homothecies,
as is the ratio between the norms. That is, if $x$ and $w$
satisfy these three conditions, and $\eta>0$, then
$x'=\eta x$ and $w'=\eta w$ also satisfy the conditions and
$\norm{w}/\norm{x}=\norm{w'}/\norm{x'}$.
Therefore we can assume that $\norm{w}=1$ since $w\neq 0$.

Suppose that the proposition is false. Then we can construct
a sequence of pairs of configurations $(x_n,w_n)$
satisfying the three conditions and such that
$\norm{w_n}=1$ for all $n\geq 0$, and 
$\lim_{n\to\infty} x_n=0$.
We can assume that
$\lim_{n\to\infty} w_n= w$, and therefore
we must have $w\in\partial\,\cone_a(\alpha)$ and $\norm{w}=1$.
Actually, because of condition (ii) we have that
\[
\inner{w}{a}=\lim_{n\to\infty}\inner{w_n}{a}=\alpha\norm{a}.
\]
But on the other hand, by condition (iii) we have that
\[
\inner{w_n-x_n}{a}\geq \alpha'\norm{w_n-x_n}\norm{a}
\]
hence
\[
\inner{w}{a}=\lim_{n\to\infty}\inner{w_n-x_n}{a}\geq \alpha'\norm{a}.
\]
This contradicts the fact that $\alpha<\alpha'$,
therefore the proposition is proved.

\end{proof}

\begin{proof}[Proof of Theorem \ref{thm:cone}]
Given $a\in\Omega$ and $\alpha>\alpha_0(a)$,
let $r_1>0$ be the constant given by
Proposition \ref{thm-unicite-loc}.   
As already remarked, it follows from
Proposition \ref{thm-unicite-loc},
and Remark \ref{rk-monot}
that if $r\geq r_1$ then for any $x\in \cone_a(\alpha,r)$
there is a unique hyperbolic motion $\gamma:[0,+\infty)\to\Omega$,
starting at $x$ at $t=0$, with limit shape $a$,
and contained in the cone $\cone_a(\alpha,r)$. 
This is exactly item (1) in Theorem \ref{thm:cone}. 
In order to prove item (2) we shall show that when
$r$ is big enough, then $\gamma$ is the unique
element in $\calM_a$ with $\gamma(0)=x$.

Let now $\beta>0$, with $\alpha_0(a)<\beta<\alpha$,
and let $r_1>0$ given by Proposition \ref{thm-unicite-loc},
but now applied to the widest cone $\cone_a(\beta,r_1)$.
Therefore we have that, for any $x\in\cone_a(\beta,r_1)$,
there exist a unique hyperbolic motion $\gamma\in\calH_a$
such that $\gamma(0)=x$ and such that $\gamma(t)\in\cone_a(\beta,r_1)$
for all $t\geq 0$. Then the proof will work as follows.

We take $x\in\cone_a(\alpha,r_1)$
and $\eta\in\calM_a$ a geodesic ray with $\gamma(0)=x$.
The existence of at least one of these geodesic rays
is proven for any $x\in E^N$ in \cite{MaVe}.
If $\eta$ is contained in $\cone_a(\beta,r_1)$,
that is, if $\eta(t)\in\cone_a(\beta,r_1)$ for all $t\geq 0$,
then $\eta$ must be the unique hyperbolic motion $\gamma$
that we know it exists under those conditions.
For this reason, what we will do next is prove that
there exists $r>r_1$ such that, if $x\in\cone_a(\alpha,r)$
then any of such geodesic rays $\eta$ must necessarily be contained
in $\cone_a(\beta,r_1)$, so $\eta=\gamma$.

Assume now, for the sake of contradiction,
that $\eta\in\calM_a$ starts at $x\in\cone_a(\alpha,r)$
at $t=0$, and is not contained in $\cone_a(\beta,r_1)$.
Since $\eta\in\calH_a$, it must enter definitively in
$\cone_a(\beta,r_1)$ and in $\cone_a(\alpha,r)$.
We denote $y$ (resp. $z$) the configuration where $\eta$
definitively enters in $\cone_a(\beta,r_1)$
(resp. $\cone_a(\alpha,r)$).
Like in the proof of Proposition
\ref{prop:borne-rentr} let us consider 
\[
\calD(x,y,z)=\phi_h(x,y)+\phi_h(y,z)-\phi_h(x,z),
\]
for $h=\norm{a}^2/2$. 
Since $\eta$ is minimizing,
it holds $\calD(x,y,z)=0$.
We show now that if $r$ is big enough,
then necessarily $\calD(x,y,z)>0$ getting a contradiction.

Remembering that both $x$ and $z$ are in $\cone_a(\alpha,r)$,
if we estimate $\phi_h(x,z)$ like in
Proposition \ref{prop:borne-rentr}
by using a path following the segment $[x,z]$
with constant speed equal to $\sqrt{2h}$, we find
\[
\phi_h(x,z)\leq\sqrt{2h}\norm{z-x} +\frac{\mu}{\sqrt{2h}\alpha}
\frac{\norm{z-x}}{r},
\]
where here the constant $\mu$ is defined as
\[
\mu=\max\set{U(w)\mid\;w\in\cone_a(\alpha),\;\norm{w}=1}.
\]
Together with the lower bounds
\[
\phi_h(x,y)\geq \sqrt{2h}\norm{x-y}
\;\;\textrm{ and }\;\; 
\phi_h(y,z)\geq \sqrt{2h}\norm{y-z}
\]
this gives the inequality 
\begin{equation} \label{siegfried}
\calD(x,y,z)\geq
\sqrt{2h}\;\calE(x,y,z)-\frac{\mu}{\sqrt{2h}\;\alpha}
\frac{\|z-x\|}{r},
\end{equation}
where $\calE:(E^N)^3\to\R$ is the function defined by 
\[
\calE(p,s,q)=\norm{s-p}+\norm{q-s}-\norm{q-p},
\quad (p,s,q)\in (E^N)^3.
\]
We look now for a lower bound of $\calE(x,y,z)$.
By definition of $y$ and $z$ we know that
$y\in\partial\cone_a(\beta,r_1)$ and
$z\in\partial\cone_a(\alpha,r)$.
If $\lambda$ is the positive constant given by
Proposition \ref{prop:borne-rentr}, we know that
$\norm{y}\geq\lambda\norm{x}\geq\lambda\,r$,
therefore, choosing $r>\lambda^{-1}r_1$ we can assume
that $\norm{y}>r_1$, which in turns implies that
$y\in\partial\cone_a(\beta)$.

On the other hand
either $\norm{z}=r$, or
$z\in\partial\cone_a(\alpha)$ with $\norm{z}>r$.
In the first case we have
\[
\norm{z}=r\leq\norm{x}\leq\lambda^{-1}\norm{y}.
\]
Otherwise, if $z\in \partial\cone_a(\alpha)$ with $\norm{z}>r$,
by applying Proposition \ref{prop:bornes-config-rentr}
we find that $\norm{z}\leq c\,\norm{y}$, where $c>0$
only depends on $r_1>0$.
In both cases, if we set $e=\max(\lambda^{-1},c)$ we can say that
$\norm{z}\leq e\,\norm{y}$ and $\norm{x}\leq e\,\norm{y}$.

Let us consider the compact subset of $(E^N)^3$ defined as
\[
\calK=\{(p,s,q)
\in\cone_a(\alpha)\times
\partial\cone_a(\beta)\times
\cone_a(\alpha)\;\mid\;
\norm{s}=1,\;
\norm{p}\leq e,
\text{ and }
\norm{q}\leq e\}.
\]
Clearly $\norm{y}^{-1}(x,y,z)\in \calK$. Therefore,
since $\calE$ is continuous and homogeneous of
degree $1$, we have that
\[
\calE(x,y,z)\geq \nu \norm{y},\quad\text{ where }\quad
\nu=\min\set{\calE(p,s,q)\,\mid\,(p,s,q)\in\calK}.
\]
The function $\calE$ is always nonnegative,
with $\calE(p,s,q)=0$ if and only if $s$ lies
on the segment joining $p$ to $q$.
However if $(p,s,q)\in\calK$, since $\cone_a(\alpha)$ is convex and $\cone_a(\alpha)\cap\partial\cone_a(\beta)=\{0\}$, then
$s$ is never
on the segment joining $p$ to $q$, thus $\nu>0$,
and moreover this constant depends only of $e$,
which in turns depends only on $r_1$. 
By replacing in (\ref{siegfried}),
and using that $x$ and $z$ are both in the ball
$\overline B(e\norm{y})$, we find
\[
\calD(x,y,z)\geq\sqrt{2h}
\left(\nu- \frac{\mu e}{h\alpha r}\right)\norm{y},
\]
which is strictly positive if $r$ is chosen so that
\[
r>r_2=\frac{\mu e }{h\alpha\nu}.
\]
We conclude that for $r>r_2$, and any $x\in\cone_a(\alpha,r)$,
there is a unique geodesic ray $\gamma\in\calM_a$
such that $\gamma(0)=x$ as we wanted to prove.

\end{proof}

\appendix
\section{Regularity of hyperbolic Viscosity solutions}
\label{ViscoSol}

Here we will prove that given a hyperbolic viscosity solution
$u:E^N\to\R$, the set of configurations $x\in\Omega$ where $u$
is differentiable is exactly the set of those configurations
for which there is a unique calibrating curve
$\gamma:(-\infty,0]\to\Omega$ arriving at $x$ at time $t=0$.

The following lemma is essentially a particular case
of Proposition 4.11.1 in the well known
and unpublished book \cite{Fa}.
We only need to adapt the argument
in order to take into account collision singularities.  
\begin{lemma} \label{lem:sur-diff}
Given $a\in\Omega$, and $u\in\calS_a$ a hyperbolic
viscosity solution with limit shape $a$,
there exist $K>0$ and $r>0$ such that
$\overline{B}(x,\delta)\subset\Omega$ and
such that, for any configurations
$y$ and $z$ in $\overline{B}(x,r)$
and any curve
$\gamma: (-\infty,0]\rightarrow \Omega$ calibrating for $u$ and satisfying 
$\gamma(0)=y$, we have
\[
u(z)-u(y)\leq \inner{\dot{\gamma}(0)}{z-y} +K\norm{z-y}^2.
\]
In particular, if $u$ is differentiable at $y$ then $d_y u (v)=\inner{\dot{\gamma}(0)}{v}$ for all $v\in E^N$. 
\end{lemma}
\begin{proof}
As usual let us set $h=\norm{a}^2/2$.
Since $\overline{B}(x,\delta)\subset \Omega$, by conservation of energy we know that there exist a constant $A>0$ such that for each solution $\sigma : I \rightarrow \overline{B}(x,\delta)$ of the $N$-body problem with energy $h$ it holds 
$\norm{\dot{\sigma}(t)}\leq A$ for all $t\in I$. Let $r\in (0,\delta)$ that we will precise later, and $\epsilon=r/A$. 
Given now $y$ and $z$ in $\overline{B}(x,r)$, if $\gamma : (-\infty,0]\rightarrow \Omega$ is a curve calibrating $u$ and arriving at $y$ at time $t=0$, let us define the  
deformation 
\[
\gamma_z : (-\infty,0]\rightarrow E^N, \quad \gamma_z(t)=\left\{
\begin{array}{ll}
\gamma(t) \quad &\text{if} \quad t\in (-\infty,-\epsilon] \\
\gamma(t)+\frac{t+\epsilon}{\epsilon}(z-y)  \quad &\text{if} \quad t\in [-\epsilon,0],
\end{array}
\right.
\]
thus $\gamma_y=\gamma$. 
By choice of $y$, $z$ and $\epsilon$ we know that $\gamma_z(t)\in \overline{B}(x,4r)$  for $t\in [-\epsilon,0]$, therefore 
we set $r=\delta/4$ and then we will have that $\gamma_z(t)\in \overline{B}(x,\delta)$ for all $t\in [-\epsilon,0]$. 
However $u$ is dominated by $L+h$, $\gamma$ is calibrating for $u$ and $\gamma_z(0)=z$ thus we have
\[
u(z)-u(y) \leq \int_{-\epsilon}^0 \left( L(\gamma_z(t),\dot{\gamma}_z(t))-L(\gamma(t),\dot{\gamma}(t))\right)\, dt.
\]
Using the fact that $\dot{\gamma}_z(t)=\dot{\gamma}(t)+\frac{z-y}{\epsilon}$, if we set $w=z-y$ we find 
\begin{equation} \label{eq:u-x-u-y}
u(z)-u(y)\leq \int_{-\epsilon}^0 \left(\frac{1}{\epsilon}\inner{\dot{\gamma}(t)}{w}+U(\gamma_z(t))-U(\gamma(t)) \right)\, dt+\frac{\norm{w}^2}{2\epsilon}.
\end{equation}
At the same time, if for $s\in [0,1]$ and $t\in [-\epsilon,0]$ we denote $\gamma_{s,z}(t)=\gamma(t)+s\frac{t+\epsilon}{\epsilon}w$, then  
\[
\begin{array}{rl}
U(\gamma_z(t))-U(\gamma(t))&=\left(\frac{t+\epsilon}{\epsilon}\right)\int_0^1 \inner{\nabla U(\gamma_{s,z}(t))}{w}\, ds \\
&=\left(\frac{t+\epsilon}{\epsilon}\right)\left(\inner{\nabla U(\gamma(t))}{w}+\calU(t,z)\right),
\end{array}
\]
where 
\[
\calU(t,z)=\int_0^1 \inner{\nabla U(\gamma_{s,z}(t))-\nabla U(\gamma(t))}{w}\, ds.
\]
Introducing the constant
\[
C=\max_{q\in \overline{B}(x,\delta)}  \|D^2U(q)\|
\]
we easily find that 
\[
|\calU(t,z)|\leq \frac{C}{2}\norm{w}^2,
\]
therefore, replacing in (\ref{eq:u-x-u-y}) gives
\[
\begin{array}{rcl}
u(z)-u(y)&\leq &\int_{-\epsilon}^0 \left(\frac{1}{\epsilon}  \inner{\dot{\gamma}(t)}{w}+\frac{t+\epsilon}{\epsilon}\inner{\nabla U(\gamma(t))}{w}    \right)\, dt \\
&+&\frac{C\epsilon^2+1}{2\epsilon}\norm{w}^2.
\end{array}
\]
Since $\gamma$ is a motion, by performing an integration by part at the second member we finally find 
\[
u(z)-u(y)\leq \inner{\dot{\gamma}(0)}{z-y}+K\norm{z-y}^2,
\]
for some strictly positive constant $K$ which depends only on $\delta$, $A$ and $C$. 

Suppose now that $u$ is differentiable at $y$. Given $v\in E^N$, by choosing $z=y+sv$ with $s>0$ small enough we find 
\[
\frac{u(y+sv)-u(y)}{s}\leq \inner{\dot{\gamma}(0)}{v}+Ks\norm{v}^2,
\]
thus $d_yu(v)\leq \inner{\dot{\gamma}(0)}{v}$. The same computation with $s<0$ small enough gives that
$d_yu(v)\geq \inner{\dot{\gamma}(0)}{v}$, thus $d_y u(v)=\inner{\dot{\gamma}(0)}{v}$.
\end{proof}
{\bf Remark.} This Lemma implies in particular that $u$ is superdifferentiable at each configuration without collisions, and if 
$\gamma : (-\infty,0]\rightarrow \Omega$ is a calibrating curves for $u$, then $\dot{\gamma}(0)$ is a supergradient of $u$ at $\gamma(0)$. 
\begin{proposition}
\label{prop:diff-unic-calibr}
  Let $u\in \calS_a$ and $x\in \Omega$. The function $u$ is
  differentiable at the point $x$ if and only if there exist a unique
  curve $\gamma : (-\infty,0]\rightarrow E^N$ that is calibrating for $u$ and such that $\gamma(0)=x$. 
 In this case we have $d_x u(v)=\inner{\dot{\gamma}(0)}{v}$ for all $v\in E^N$, or equivalently $\nabla u(x)=\dot{\gamma}(0)$. 
\end{proposition}
\begin{proof}
Suppose $u$ differentiable at $x$, and let $\gamma : (-\infty,0]\rightarrow \Omega$ be a calibrating curve for $u$ such that $\gamma(0)=x$.
By Lemma \ref{lem:sur-diff} we know that $\dot{\gamma}(0)=\nabla u(x)$, thus for all $t\leq 0$ we have that $\gamma(t)=\varphi^t(x,\nabla u(x))$. This show the uniqueness of $\gamma$. 

Conversely, let $x\in\Omega$ such that there is a unique  curve $\gamma : (-\infty,0]\rightarrow \Omega$ which calibrates $u$ and arrives at $x$ at $t=0$, and let us set $w=\dot{\gamma}(0)$. Consider now the multivalued function $\calV : \Omega \rightarrow \calP(E^N)$ defined in the following way : if $y\in\Omega$ then $\calV(y)$ is the set of all $p\in E^N$ such that $t\mapsto \varphi^t(y,p)$ is well defined for all $t\leq 0$ and is calibrating for $u$. We now prove the following claim.

{\it Claim.} $\lim_{y\to x} \calV(y)=\{w\}$, or equivalently, for each $\epsilon>0$ there is $\delta>0$ such that if $y\in \overline{B}(x,\delta)$ 
then $\calV(y)\subset \overline{B}(w,\epsilon)$. 

To prove this claim we will show that a limit of calibrating curve is still a calibrating curve for $u$. Let $\delta_0>0$ be such that 
$\overline{B}(x,\delta_0)\subset \Omega$. Since each calibrating curve has energy equal to $h=\norm{a}^2/2$, there is a constant $A>0$ such that for each $y\in \overline{B}(x,\delta_0)$ and $p\in \calV(y)$ we have that $\norm{p}\leq A$. Let us suppose, for the sake of contradiction, that there exist $\epsilon>0$ and a sequence $((y_n,p_n))_n$, in $\cup_{y\in \overline{B}(x,\delta_0)} \{y\}\times \calV(y)$ such that 
$\lim_{n\to +\infty} y_n=x$ and $\norm{p_n-w}>\epsilon$ for all $n\in\N$. Up to extracting a subsequence, we can assume that $(p_n)_n$ converge to 
some $p\in E^N$ satisfying $\norm{p-w}\geq \epsilon$. Let now $\sigma_n : (-\infty,0]\rightarrow \Omega$ be the calibrating curve defined by $\sigma_n(t)=\varphi^t (y_n,p_n)$, and let $(-T,0]$ be the maximal interval of existence of the motion $\sigma(t)=\varphi^t(x,p)$ for negative times. 
Since $(\sigma_n)_n$ converges to $\sigma$ in the $\calC^1$ topology on every segment contained in $(-T,0]$, we know that $\sigma$ is still a calibrating curve for $u$. 
In particular, by Remark \ref{rem:calibrat-implique-geod-min} we know that for each segment $[c,d]\subset (-T,0]$ 
it holds  
\[
\phi_h(\sigma(c),\sigma(d))=\calA_{L+h}(\sigma\left|_{[c,d]}\right.) .
\]
Let us show now that $T=+\infty$. It is essentially the same argument that is used in the proof of Theorem 3.2. in \cite{MaVe}. We write it here for the sake of clarity. Suppose by contradiction that $T<+\infty$. We consider two cases. If $\norm{\sigma(t)}$ is bounded as $t\to -T^+$, then by a classical result of von Zeipel \cite{Zeip} the limit $x_c=\lim_{t\to -T^+} \sigma(t)$ exists and is a configuration with collisions, thus $\sigma(t)$ has a collision singularity as $t\to -T^+$. In this case we extend $\sigma$ on $[-T,0]$ by setting $\sigma(-T)=x_c$. By Sundman's inequalities for total or partial collisions (see for instance \cite{Sper}), there exists a constant $\mu>0$ such that
\[
\norm{\dot{\sigma}(t)}\sim \frac{2\mu}{3} |t+T|^{-1/3}\quad \text{and} \quad U(\sigma(t))\sim \frac{2\mu^2}{9}|t+T|^{-2/3}\quad \text{as}\ t\to -T^+.
\]
Therefore the function $t \to \calA_{L+h}(\sigma\left|_{[t,0]}\right.)$ is continuous at  $t=-T$ and $\sigma$ is calibrating for $u$ in 
the full interval $[-T,0]$. By definition of hyperbolic viscosity solution we know that there is another calibrating curve 
$\eta : (-\infty,0]\to E^N$ for $u$ such that $\eta(0)=\sigma(-T)$, and by Lemma 2.15 in \cite{MaVe}, the concatenation of $\eta$ and $\sigma$, here denoted $\eta*\sigma$, is still a calibrating curve for $u$. By Marchal's Theorem  \cite{Mar} we know that $\eta*\sigma$ is without collision and a genuine solution of the $N$-body problem with energy $h$. 
This gives a contradiction, because 
$\eta(0)=\sigma(-T)=x_c$ is a configuration with collisions.  

On the other hand, if $\sigma(t)$ is unbounded as $t\to -T^+$, there exists a sequence $(t_n)_n$ in $(-T,0]$ converging to $-T$ and  such that if we set $q_n=\sigma(t_n)$ then we have $\norm{q_n-x}\to +\infty$ as $n\to +\infty$. Let us set $A_n=\calA(\sigma\left|_{[t_n,0]}\right.)$. By neglecting the potential term in the integral defining the action and by using Cauchy-Schwarz inequality we find $2|t_n| A_n \geq \norm{q_n-x}^2$, therefore
\[
\phi_h(q_n,x)=A_n+h|t_n|\geq\frac{\norm{q_n-x}^2}{2|t_n|}+h|t_n| 
\]
Since $(t_n)_n$ is bounded and $(q_n)_n$ is unbounded, for $n$ big enough this contradicts the inequality proved in \cite{MaVe}, Theorem 2.11. 

We have proved that $T=+\infty$, that is to say $\sigma$ is a global calibrating curve for $u$. Since $\norm{p-w}\geq \epsilon$, we know that $\sigma\neq \gamma$.  This contradicts the uniqueness of the calibrating curve arriving at $x$ and achieve the proof of the claim.    

Let now $\epsilon>0$. By the claim that we have just proved we can find $\delta>0$ such that $\overline{B}(x,\delta)\subset \Omega$ and for all $y\in \overline{B}(x,\delta)$ and $p\in \calV(y)$ it holds $\norm{p-w}<\epsilon$. By applying Lemma \ref{lem:sur-diff} we find two constants $K>0$ and $r\in (0,\delta)$ such that for any $y\in \overline{B}(x,r)$ and  $p\in\calV(y)$ it holds
\[
u(y)-u(x)\leq \inner{w}{y-x}+K\norm{y-x}^2\quad \text{and} \quad u(x)-u(y)\leq \inner{p}{x-y}+K\norm{y-x}^2,
\]
and therefore 
\[
-\epsilon\norm{y-x}-K\norm{y-x}^2   \leq u(y)-u(x)-\inner{w}{y-x}\leq K\norm{y-x}^2.
\]
This last inequality can be written as
\[
u(y)=u(x)+\inner{w}{y-x}+o(\norm{y-x}),\quad y\to x,
\]
thus $u$ is differentiable at $x$ and $\nabla u(x)=w=\dot{\gamma}(0)$.
\end{proof}

\section{More on hyperbolic motions}
\label{app:More-Hyp}

Hyperbolic motions are in the class of expansive motions,
that is, the class of motions in which all mutual distances diverge.
For the general $N$-body problem,
a classification of the possible final evolutions for
expansive motions was largely described a century ago
in the work of J.~Chazy \cite{Cha1,Cha2}. 
A more complete and refined description
was obtained fifty years later, in particular
in the contributions of H.~Pollard \cite{Pol}, and
Ch.~Marchal \& D.~Saari \cite{MarSaa}.

It turns out that, in any expansive motion,
the energy constant must be $h\geq 0$
and each distance function
$r_{ij}(t)$ must diverge like $t$ or like $t^{2/3}$ as $t\to +\infty$.
If the motion is expansive but not hyperbolic, then either $h>0$
and some distances expand like $t^{2/3}$,
or $h=0$ and all mutual distances grow like $t^{2/3}$.
In the first case we say that the motion is
\emph{partially hyperbolic},
and as well as in the hyperbolic case, it admits
a limit shape $a=\lim_{t\to +\infty}\,(x(t)/t)$, but
unlike the hyperbolic case, the limit $a$ has collisions,
because $a_i=a_j$ for any pair $i<j$ such that
$r_{ij}(t)\approx t^{2/3}$.

In the second case, in which $h=0$, the expansivity
is also characterized by the property
$\lim_{t\to +\infty}\dot x(t)=0$.
Such motions are called \emph{parabolic},
and in general it is not possible to guarantee
the existence of a limit shape
$c= \lim_{t\to +\infty}\,\left[x(t)/t^{2/3}\right]$.
In case of existence, the limit shape of a parabolic motion
must necessarily be a central configuration.
Actually, it is well known that for a parabolic motion,
the quotient $x(t)/t^{2/3}$ approximates the set of central configurations,
and that the limit must exist in the case of approximating
non-degenerate central configurations.

The importance of hyperbolic motions for the study of the dynamics
of the $N$-body problem lies, at least, in the following
two reasons.

The first one is that, as we have shown in Lemma \ref{lem:lema-cont.limitshape}, the set of initial conditions
giving rise to hyperbolic motions is open
in the phase space. Moreover
the limit shape, as a function
defined on this open set, is smooth and invariant by the flow.
This says that hyperbolic motions constitute both the simplest
and strongest form of diffusion.
The subject has been extensively studied,
as well in the classical $N$-body problem,
as in many other physical models of particle interactions,
see for instance the well known paper by Simon on
classical particle scattering \cite{Sim}.
Even if the literature on the subject is extremely vast,
many natural issues seem to be far from being well understood,
and fundamental aspects are still being developed.
A clear example of this is the recent work by
Féjoz, Knauf and Montgomery \cite{FeKM}
in which they establish, for very general potentials,
a conjugation in the free part of the phase space,
between the considered dynamics
and the dynamics of free particles.

The second reason we want to mention is that they are abundant,
in the sense that they can be found for every initial configuration,
even for any prescribed limit shape and energy constant.
This is in fact the main result in our previous work \cite{MaVe}.
Moreover, the subset of the phase space giving rise to
bi-hyperbolic motions, that is to say,
entire motions which are hyperbolic
both for $t\to +\infty$ and for $t\to -\infty$, is also a non-empty
open set. This allows to define a scattering relation between
shapes, namely, two configurations without collisions are related
when they are, respectively, the past and the future limit shape
for some bi-hyperbolic motion.
This question has been studied by Duignan, Moeckel, Montgomery
and Yu \cite{DMMY}, who proved that,
given any configuration for the scattering shape in the past,
the set of related scattering shapes for the future
always contains an open subset. For this the authors perform
a partial compactification of the phase space by adding an
invariant manifold at infinity,
in such a way that the hyperbolic orbits are embedded
in the stable manifolds of equilibrium points in that manifold.
Nevertheless,
even for the planar three body problem it is not known if any pair
of configurations are scattering related
or if some pairs are excluded.
Moreover, it is unknown if in the three body problem
the set of related pairs is dense
like in the Kepler problem.

\vspace{.5cm}
\textbf{Acknowledgments.}
We want to thank Albert Fathi and Maxime Zavidovique
for several interesting discussions
that allowed this text to be greatly improved. The authors thank the anonymous referee for his careful reading 
and for his helpful comments that improved the quality of the manuscript. 

\vspace{.5cm}
\textbf{Statements and Declarations.}
The authors are not aware of any conflict of interest.
This research has no associated data.

%%%%%%%%%%%%%%%%%%%%%%%%%%%%%%%%%%%%%%%%%%%%%%%%%%%%%%%%%%%%%%%%%%
%%%%%%%%%%%%%%%%%%%%%%%%%%%%%%%%%%%%%%%%%%%%%%%%%%%%%%%%%%%%%%%%%%
%%%%%%%%%%%%%%%%%%%%%%%%%%%%%%%%%%%%%%%%%%%%%%%%%%%%%%%%%%%%%%%%%%
%%%%%%%%%%%%%%%%%%%%%%%%%%%%%%%%%%%%%%%%%%%%%%%%%%%%%%%%%%%%%%%%%%
%%%%%%%%%%%%%%%%%%%%%%%%%%%%%%%%%%%%%%%%%%%%%%%%%%%%%%%%%%%%%%%%%%
%%%%%%%%%%%%%%%%%%%%%%%%%%%%%%%%%%%%%%%%%%%%%%%%%%%%%%%%%%%%%%%%%%
%%%%%%%%%%%%%%%%%%%%%%%%%%%%%%%%%%%%%%%%%%%%%%%%%%%%%%%%%%%%%%%%%%
%%%%%%%%%%%%%%%%%%%%%%%%%%%%%%%%%%%%%%%%%%%%%%%%%%%%%%%%%%%%%%%%%%

\end{document}